\documentclass[preprint,11pt,a4paper,oneside,reqno]{elsarticle}%
\usepackage{amsmath}
\usepackage{amsfonts}
\usepackage{amssymb}
\usepackage{graphicx}
\usepackage[margin=0.9in]{geometry}%
\providecommand{\U}[1]{\protect\rule{.1in}{.1in}}
\newtheorem{theorem}{Theorem}

\newtheorem{case}{Case}

\newtheorem{corollary}{Corollary}

\newtheorem{definition}{Definition}
\newtheorem{example}{Example}

\newtheorem{lemma}{Lemma}

\newtheorem{remark}{Remark}

\newproof{proof}{Proof}
\numberwithin{equation}{section}
\begin{document}
\begin{frontmatter}
\title{Floquet theory based on new periodicity concept for hybrid systems involving $q$-difference equations\tnoteref{t1}}
\tnotetext[t1]{This study is supported by The Scientific and Technological Research Council
of Turkey (grant no. 1649B031101152).}
\author{Murat Ad\i var\corref{cor1}}
\ead{murat.adivar@ieu.edu.tr}
\address{ Izmir University of Economics\\
Department of Mathematics, 35330, Izmir Turkey}
\cortext[cor1]{Corresponding author}
\author{Halis Can Koyuncuo\u{g}lu}
\ead{can.koyuncuoglu@ieu.edu.tr}
\address{ Izmir University of Economics\\
Department of Mathematics, 35330, Izmir Turkey}

\begin{abstract}
Using the new periodicity concept based on shifts, we construct a unified Floquet theory for homogeneous and
nonhomogeneous hybrid periodic systems on domains having continuous, discrete or hybrid structure. New periodicity concept based on shifts enables the construction of Floquet theory on hybrid domains that are not
necessarily additive periodic. This makes periodicity and stability analysis of hybrid periodic systems possible on non-additive domains. In particular, this new
approach can be useful to know more about Floquet theory for linear
$q$-difference systems defined on $\overline{q^{\mathbb{Z}}}:=\{q^{n}%
:n\in\mathbb{Z}\}\cup\{0\}$ where $q>1$. By constructing the solution of
matrix exponential equation we establish a canonical Floquet decomposition
theorem. Determining the relation between Floquet multipliers and Floquet
exponents, we give a spectral mapping theorem on closed subsets of reals that
are periodic in shifts. Finally, we show how the constructed theory can be utilized for the stability analysis
of dynamic systems on periodic time scales in shifts.

\end{abstract}
\end{frontmatter}

\section{Introduction}

The theory of periodic systems has taken a
prominent attention in the existing literature due to its tremendous application potential in
engineering, biology, biomathematics, chemistry etc. Floquet theory is an important tool for the investigation of periodic solutions and stability analysis
of dynamic systems. Floquet theory of differential and difference systems can
be found in \cite{[13]} and \cite{[14]}, respectively. Floquet theory of
Volterra equation has been handled in \cite{[7]}. An extension of the Floquet theory
to the systems with memory has been studied in \cite{Belbas}. In \cite{[6]},
Floquet theory has been employed for stability analysis of nonlinear
integro-differential equations. Moreover, a generalization of Floquet
theory in continuous case is studied in \cite{generalized}.

Providing a wide perspective to discrete and continuous analysis, time scale
calculus is a useful theory for the unification of differential and
difference systems. For the sake of brevity, we suppose familiarity
with fundamental theory of time scales. For a comprehensive review on time scale theory, we may
refer readers to \cite{[1]} and \cite{[2]}. Unification of discrete and continuous dynamic systems under the theory of
time scales avoids the separate studies for differential and difference systems by using the similar
arguments. Motivated by unification and extension capabilities of time scale calculus, the researchers in recent years have been developing the time scale analogues of existing results for difference, $q-$difference, and
differential equations. For instance in \cite{[5]}, the authors construct a
Floquet theory for additive periodic time scales and focus on Putzer
representations of matrix logarithms. We use the terminology "additive periodic time
scale" to refer to an arbitrary, closed, non-empty subset $\mathbb{T}$ of reals
satisfying the following property (\cite{Kaufmann}):%

\begin{equation}
\text{there exists a fixed }P\in\mathbb{T}\text{ such that }t\pm P\in\mathbb{T}%
\text{ for all }t\in\mathbb{T}\text{.}\label{P}%
\end{equation}
In \cite{[4]}, DaCunha unified Floquet theory for nonautonomous linear dynamic
systems based on Lyapunov transformations by using matrix exponential on time
scales (see \cite[Section 5]{[1]}). Afterwards, DaCunha and Davis improved the
results of \cite{[4]} in \cite{[3]}. Note that the results in \cite{[3]} and \cite{[4]} regarding Floquet theory are valid only on additive periodic time scales. This
strong restriction prevents investigation of periodicity on very important particular time
scales. For instance, the $q$-difference equations are established on the time
scale%
\[
\overline{q^{\mathbb{Z}}}:=\left\{  q^{n}:n\in\mathbb{Z}\right\}  \cup\left\{
0\right\}  \text{, }q>1
\]
which is not additive periodic. Hence, the existing unified Floquet theory does not cover the systems of $q$-difference equations. A $q$-difference equation is an
equation including a $q$-derivative $D_{q}$, given by%
\[
D_{q}\left(  f\right)  \left(  t\right)  =\frac{f\left(  qt\right)  -f\left(
t\right)  }{\left(  q-1\right)  t}\text{,\ \ }t\in q^{\mathbb{Z}}\text{,}%
\]
of its unknown function. Observe that the $q-$derivative $D_{q}\left(
f\right)  $ of a function $f$ turns into ordinary derivative $f^{\prime}$ if
we let $q\rightarrow1$. The theory of $q$-difference equations is a useful tool for the discretization of differential equations used for modeling continuous processes
 (see e.g. \cite{dobrogovska}, \cite{Malkiewicz},
\cite{ostrovska}, and references therein). In \cite{pulita} the author says "\textit{in the }$p-$\textit{adic context, }$q$%
\textit{-difference equations are not simply a discretization of solutions of
differential equations, but they are actually equal"}. We may also refer to
\cite{andre} for further discussion about the equivalence between
$q$-difference equations and differential equations. There is a vast literature on the existence of periodic solutions of differential equations, unlike the existence of periodic solutions of $q$-difference equations. Thus, it is of importance to study the existence of periodic solutions of $q$-difference equations.

In recent years, the shift operators, denoted $\delta_{\pm}\left(  s,t\right)
$, are introduced to construct delay dynamic equations and a new periodicity
concept on time scales (see \cite{[8]},
\cite{shift}, and \cite{[11]}). We give a detailed information about the shift
operators in further sections. We may also refer to the studies \cite{[9]},
\cite{[10]}, and \cite{[11]}  for the basic definitions, properties and some
applications of shift operators on time scales. In particular, we direct the
readers to \cite{[8]} for the construction of new periodicity concept on time
scales. The motivation of new periodicity concept in \cite{[8]} stems from the
following ideas:

\begin{enumerate}
\item[I.1.] Addition is not always the only way to step forward and backward
on a time scale, for instance, the operators $\delta_{\pm}\left(
2,t\right)  =2^{\pm1}t$ determine backward and forward shifts on the
time scale $\left\{  2^{n}:n\in\mathbb{Z}\right\}  \cup\left\{  0\right\}  $

\item[I.2.] We may use shift operators $\delta_{\pm}$ with certain properties to obtain a backward and forward motion on a general time scale. Similar to (\ref{P}) a periodic time scale in shifts can be
defined to be the one satisfying the following property:%

\begin{equation}
\text{there exists a fixed }P\in\mathbb{T}\text{ such that }\delta_{\pm}\left(  P,t\right)  \in\mathbb{T}\text{ for all }t\in\mathbb{T}.%
\end{equation}

%
\end{enumerate}

This approach enables the study of periodicity notion on a large class of time scales that are not necessarily additive periodic. For instance, the time
scale $\overline{q^{\mathbb{Z}}}$ is periodic in shifts $\delta_{\pm}\left(
s,t\right)  =s^{\pm}t$ since
\[
\delta_{\pm}\left(  q,t\right)  =q^{\pm}t\in\mathbb{T}\text{ for all }%
t\in\mathbb{T}\text{.}%
\]
Therefore, one may define a $q^{k}$-periodic function $f$ on $\overline
{q^{\mathbb{Z}}}$ as follows:%
\[
f\left(  q^{\pm k}t\right)  =f\left(  t\right)  \text{ for all }t\in
\overline{q^{\mathbb{Z}}}\text{ and a fixed }k\in\left\{  1,2,\ldots\right\}
.
\]
More generally, a $T$-periodic function $f$ on a $P$-periodic time scale
$\mathbb{T}$ in shifts $\delta_{\pm}$ can be defined as follows%
\[
f\left(  \delta_{\pm}\left(  T,t\right)  \right)  =f\left(  t\right)  \text{
for all }t\in\mathbb{T}\text{ and a fixed }T\in\left[  P,\infty\right)
\cap\mathbb{T}\text{.}%
\]

In this paper, we use Lyapunov transformation (see \cite[Definition 2.1]{[3]})
and the new periodicity concept developed in \cite{[8]} to construct a unified
Floquet theory for hybrid systems on hybrid domains. As an alternative to the existing literature,
our Floquet theory and stability results are valid on more time scales, such
as $\overline{q^{\mathbb{Z}}}$ and%
\[
\cup_{k=1}^{\infty}\left[  3^{\pm k},2.3^{\pm k}\right]  \cup\left\{
0\right\}
\]
which cannot be covered by \cite{[3]} and \cite{[4]}. It should be mentioned
that periodicity notion and Floquet theory on the time scale%
\[
q^{\mathbb{N}_{0}}=\left\{  q^{n}:q>1\text{ and }n=0,1,2,\ldots\right\}
\]
have been studied in \cite{[sarajevo]} and \cite{[15]}. In \cite{[sarajevo]} and \cite{[15]} a
$\omega$-periodic function $f$ on $q^{\mathbb{N}_{0}}$ is defined to be the one
satisfying%
\[
f\left(  q^{\omega}t\right)  =\frac{1}{q^{\omega}}f\left(  t\right)  \text{
for all }t\in q^{\mathbb{N}_{0}}\text{ and a fixed }\omega\in\left\{
1,2,\ldots\right\}  .
\]
According to this periodicity definition the function $g\left(  t\right)
=1/t$ is $q$-periodic over the time scale $q^{\mathbb{N}_{0}}$. Unlike the conventional periodic functions in the existing literature, the function
$g\left(  t\right)  =1/t$ does not repeat its values at each period
$t,q^{\omega}t,\left(  q^{\omega}\right)  ^{2}t,...$. In parallel with conventional
periodicity perception, we define a periodic function to be the one repeating its values
at each forward/backward step on its domain with a certain size. For instance, according to our definition the
function $h(t)=\left(  -1\right)  ^{\frac{\ln t}{\ln q}}$ is a $q^{2}%
$-periodic function on $q^{\mathbb{Z}}=\left\{  q>1:q^{n},n\in\mathbb{Z}%
\right\}  $ since%
\[
h\left(  \delta_{\pm}\left(  q^{2},t\right)  \right)  =\left(  -1\right)
^{\frac{\ln t}{\ln q}\pm2}=\left(  -1\right)  ^{\frac{\ln t}{\ln q}}=h\left(
t\right)  .
\]
Obviously, the function $h\left(  t\right)  $ repeats the values $-1$ and $1$
at each backward/forward step with the size $q^{2}$. Consequently, the use
of new periodicity concept based on shifts $\delta_{\pm}$ in Floquet theory
provides not only a generalization but also an alternative approach to already
existing literature in particular cases (e.g. \cite{[sarajevo]} and \cite{[15]}).

We organize the rest of the paper as follows: In Section 2, we introduce the basic concepts and in Section 3 we develop Floquet theory based on new periodicity concept on time scales. We end the paper by applying our results to stability analysis of linear systems.

\section{Preliminaries}

\subsection{Matrix exponential}

In this section we give some basic definitions and results that we require in
our further analysis.

A time scale, denoted by $\mathbb{T}$, is an arbitrary, nonempty and\ closed
subset of real numbers. A time scale may have a discrete or connected
structure as well as a hybrid structure consisting of intervals and isolated
points. The operator $\sigma:$ $\mathbb{T\rightarrow T}$ called forward jump
operator is defined by $\mathbb{\mathbb{\sigma}}\left(  t\right)
:=\inf\left\{  s\in\mathbb{T},s>t\right\}  $. The step size function
$\mu:\mathbb{T\rightarrow R}$ is given by $\mu\left(  t\right)  :=\sigma
\left(  t\right)  -t$. We say a point $t\in\mathbb{T}$ is right dense if
$\mu\left(  t\right)  =0$, and right scattered if $\mu\left(  t\right)  >0$.
Furthermore, a point $t\in\mathbb{T}$ is said to be left dense if $\rho\left(
t\right)  :=\sup\left\{  s\in\mathbb{T},s<t\right\}  =t$ and left scattered if
$\rho\left(  t\right)  <t$. A function $f:\mathbb{T}\rightarrow\mathbb{R}$ is
said to be $rd$-continuous if it is continuous at right dense points and its
left sided limits exists at left dense points. The set $\mathbb{T}^{k}$ is
defined in the following way: If $\mathbb{T}$ has a left-scattered maximum
$m$, then $\mathbb{T}^{k}=\mathbb{T}-\left\{  m\right\}  ;$ otherwise
$\mathbb{T}^{k}=\mathbb{T}.$ Moreover, the delta derivative of a function
$f:\mathbb{T}\rightarrow\mathbb{R}$ at a point $t\in\mathbb{T}^{k}$ is defined
by%
\[
f^{\Delta}\left(  t\right)  :=\lim_{\substack{s\rightarrow t\\s\neq
\sigma\left(  t\right)  }}\frac{f\left(  \sigma\left(  t\right)  \right)
-f\left(  s\right)  }{\sigma\left(  t\right)  -s}.
\]

\begin{definition}
\label{def2.1}A function $p:\mathbb{T\rightarrow R}$ is said to be regressive
if $1+\mu\left(  t\right)  p\left(  t\right)  \neq0$ for all $t\in
\mathbb{T}^{k}.$ We denote by $\mathcal{R}$ the set of all regressive functions.
\end{definition}

\begin{definition}
[Exponential function]Let $\varphi\in\mathcal{R}$ and $\mu(t)>0$ for all
$t\in\mathbb{T}$. The \emph{exponential function} on $\mathbb{T}$ is defined
by
\[
e_{\varphi}(t,s)=\exp\left(  \int_{s}^{t}\!\frac{1}{\mu(z)}\mbox{Log}(1+\mu
(z)\varphi(z))\,\Delta z\right)  .
\]

\end{definition}

It is well known that if $p\in\mathcal{R}^{+}$, then $e_{p}(t,s)>0$ for all
$t\in\mathbb{T}$. Also, the exponential function $y(t)=e_{p}(t,s)$ is the
solution to the initial value problem $y^{\Delta}=p(t)y,\,y(s)=1$. Other
properties of the exponential function are given in the following lemma:

\begin{lemma}
\cite[Theorem 2.36]{[1]} Let $p,q\in\mathcal{R}$. Then

\begin{itemize}
\item[i.] $e_{0}(t,s)\equiv1$ and $e_{p}(t,t)\equiv1$;

\item[ii.] $e_{p}(\sigma(t),s)=(1+\mu(t)p(t))e_{p}(t,s)$;

\item[iii.] $\frac{1}{e_{p}(t,s)}=e_{\ominus p}(t,s)$ where, $\ominus
p(t)=-\frac{p(t)}{1+\mu(t)p(t)}$;

\item[iv.] $e_{p}(t,s)=\frac{1}{e_{p}(s,t)}=e_{\ominus p}(s,t)$;

\item[v.] $e_{p}(t,s)e_{p}(s,r)=e_{p}(t,r)$;

\item[vi.] $\left(  \frac{1}{e_{p}(\cdot,s)}\right)  ^{\Delta}=-\frac
{p(t)}{e_{p}^{\sigma}(\cdot,s)}$.
\end{itemize}
\end{lemma}

\begin{definition}
[Matrix exponential]\label{def2.2}\cite[Definition 5.18]{[1]} Let $t_{0}%
\in\mathbb{T}$ and assume that $A\in\mathcal{R}$ is an $n\times n$
matrix-valued function. The unique matrix solution of the IVP%
\[
Y^{\Delta}\left(  t\right)  =A\left(  t\right)  Y,\text{ }Y\left(
t_{0}\right)  =I,
\]
where $I$ denotes as usual $n\times n$ identity matrix, is called the matrix
exponential function, and is denoted by $e_{A}\left(  .,t_{0}\right)  .$
\end{definition}

\begin{theorem}
\label{thm2.1}\cite[Theorem 5.21]{[1]} Let $A,B\in\mathcal{R}$ be $n\times n$
matrix-valued functions on time scale $\mathbb{T}$, then we have
\end{theorem}

\begin{enumerate}
\item $e_{0}\left(  t,s\right)  \equiv I$ and $e_{A}\left(  t,t\right)  \equiv
I$, where $0$ and $I$ indicate the zero matrix and the identity matrix, respectively,

\item $e_{A}\left(  \sigma\left(  t\right)  ,s\right)  =\left(  I+\mu\left(
t\right)  A\left(  t\right)  \right)  e_{A}\left(  t,s\right)  ;$

\item $e_{A}\left(  t,s\right)  =e_{A}^{-1}\left(  s,t\right)  ;$

\item $e_{A}\left(  t,s\right)  e_{A}\left(  s,r\right)  =e_{A}\left(
t,r\right)  ;$

\item $e_{A}\left(  t,s\right)  e_{B}\left(  t,s\right)  =e_{A\oplus B}\left(
t,s\right)  $, where
\[
\left(  A\oplus B\right)  \left(  t\right)  =A\left(  t\right)  +B\left(
t\right)  +\mu\left(  t\right)  A\left(  t\right)  B\left(  t\right)  .
\]

\end{enumerate}

\begin{theorem}
\label{thm2.2}\cite[Theorem 5.24]{[1]}(variation of constants). Let
$A\in\mathcal{R}$ be an $n\times n$ matrix-valued function on $\mathbb{T}$ and
suppose that $f:\mathbb{T}\rightarrow\mathbb{R}^{n}$ is $rd$-continuous. Let
$t_{0}\in\mathbb{T}$ and $y_{0}\in\mathbb{R}^{n}.$ Then the initial value
problem%
\[
y^{\Delta}=A\left(  t\right)  y+f\left(  t\right)  ,\text{ }y\left(
t_{0}\right)  =y_{0}%
\]
has a unique solution $y:\mathbb{T}\rightarrow\mathbb{R}^{n}.$ Moreover, this
solution is given by%
\[
y\left(  t\right)  =e_{A}\left(  t,t_{0}\right)  y_{0}+%
{\displaystyle\int\limits_{t_{0}}^{t}}
e_{A}\left(  t,\sigma\left(  \tau\right)  \right)  f\left(  \tau\right)
\Delta\tau.
\]

\end{theorem}

\subsection{Shift operators and new periodicity concept \ based on shift
operators}

In this section, we aim to introduce basic definitions and properties of shift
operators. The following definitions, lemmas and examples can be found in
\cite{[8]}, \cite{[9]}, \cite{[10]} and \cite{[11]}.

\begin{definition}[Shift operators]
\label{def3.1}Let $\mathbb{T}^{\ast }$ be a nonempty subset of the time
scale $\mathbb{T}$ including a fixed number $t_{0}\in \mathbb{T}^{\ast }$
such that there exists operators $\delta _{\pm }:\left[ t_{0},\infty \right)
_{\mathbb{T}}\times \mathbb{T}^{\ast }\rightarrow \mathbb{T}^{\ast }$
satisfying the following properties:

\begin{enumerate}
\item \label{def3.1_1}The functions $\delta _{\pm }$ are strictly increasing
with respect to their second arguments, if%
\begin{equation*}
\left( T_{0},t\right) ,\left( T_{0},u\right) \in \mathcal{D}_{\pm }:=\left\{
\left( s,t\right) \in \left[ t_{0},\infty \right) _{\mathbb{T}}\times
\mathbb{T}^{\ast }:\delta _{\pm }\left( s,t\right) \in \mathbb{T}^{\ast
}\right\} ,
\end{equation*}%
then%
\begin{equation*}
T_{0}\leq t<u\text{ implies }\delta _{\pm }\left( T_{0},t\right) <\delta
_{\pm }\left( T_{0},u\right) ;
\end{equation*}

\item If $\left( T_{1},u\right) ,\left( T_{2},u\right) \in \mathcal{D}_{-}$
with $T_{1}<T_{2},$ then $\delta _{-}\left( T_{1},u\right) >\delta
_{-}\left( T_{2},u\right) $ and if $\left( T_{1},u\right) ,\left(
T_{2},u\right) \in \mathcal{D}_{+}$ with $T_{1}<T_{2},$ then $\delta
_{+}\left( T_{1},u\right) <\delta _{+}\left( T_{2},u\right) ;$

\item If $t\in \left[ t_{0},\infty \right) _{\mathbb{T}},$ then $\left(
t,t_{0}\right) \in \mathcal{D}_{+}$ and $\delta _{+}\left( t,t_{0}\right)
=t. $ Moreover, if $t\in \mathbb{T}^{\ast },$ then $\left( t_{0},t\right)
\in \mathcal{D}_{+}$ and $\delta _{+}\left( t_{0},t\right) =t;$
\item
\begin{enumerate}
\item If $\left( s,t\right) \in \mathcal{D}_{+},$ then $\left( s,\delta
_{+}\left( s,t\right) \right) \in \mathcal{D}_{-}$ and $\delta _{-}\left(
s,\delta _{+}\left( s,t\right) \right) =t;$

\item If $\left( s,t\right) \in \mathcal{D}_{-},$ then $\left( s,\delta
_{-}\left( s,t\right) \right) \in \mathcal{D}_{+}$ and $\delta _{+}\left(
s,\delta _{-}\left( s,t\right) \right) =t;$
\end{enumerate}
\item
\begin{enumerate}
\item If $\left( s,t\right) \in \mathcal{D}_{+}$ and $\left( u,\delta
_{+}\left( s,t\right) \right) \in \mathcal{D}_{-},$ then $\left( s,\delta
_{-}\left( u,t\right) \right) \in \mathcal{D}_{+}$ and $\delta _{-}\left(
u,\delta _{+}\left( s,t\right) \right) =\delta _{+}\left( s,\delta
_{-}\left( u,t\right) \right) ;$

\item If $\left( s,t\right) \in \mathcal{D}_{-}$ and $\left( u,\delta
_{-}\left( s,t\right) \right) \in \mathcal{D}_{+},$ then $\left( s,\delta
_{+}\left( u,t\right) \right) \in \mathcal{D}_{-}$ and $\delta _{+}\left(
u,\delta _{-}\left( s,t\right) \right) =\delta _{-}\left( s,\delta
_{+}\left( u,t\right) \right) .$
\end{enumerate}
\end{enumerate}

\noindent Then the operators $\delta _{+}$ and $\delta _{-}$ are called forward and
backward shift operators associated with the initial point $t_{0}$ on $%
\mathbb{T}^{\ast }$ and the sets $\mathcal{D}_{+}$ and $\mathcal{D}_{-}$ are
domain of the operators, respectively.
\end{definition}

\begin{example}
\label{rem 3.1}The following table shows the shift operators $\delta_{\pm
}\left(  s,t\right)  $ on some time scales:%
\[%
\begin{tabular}
[c]{|c|c|c|c|c|}\hline
$\mathbb{T}$ & $t_{0}$ & $\mathbb{T}^{\ast}$ & $\delta_{-}\left(  s,t\right)
$ & $\delta+\left(  s,t\right)  $\\\hline
$\mathbb{R}$ & $0$ & $\mathbb{R}$ & $t-s$ & $t+s$\\\hline
$\mathbb{Z}$ & $0$ & $\mathbb{Z}$ & $t-s$ & $t+s$\\\hline
$q^{\mathbb{Z}}\cup\left\{  0\right\}  $ & $1$ & $q^{\mathbb{Z}}$ & $\frac
{t}{s}$ & $st$\\\hline
$\mathbb{N}^{1/2}$ & $0$ & $\mathbb{N}^{1/2}$ & $\left(  t^{2}-s^{2}\right)
^{1/2}$ & $\left(  t^{2}+s^{2}\right)  ^{1/2}$\\\hline
\end{tabular}
\ \ .
\]

\end{example}

\begin{lemma}
\label{lem3.1}Let $\delta_{\pm}$ be the shift operators associated with the
initial point $t_{0}.$ Then we have the following:
\end{lemma}

\begin{enumerate}
\item $\delta_{-}\left(  t,t\right)  =t_{0}$ for all $t\in\left[  t_{0}%
,\infty\right)  _{\mathbb{T}};$

\item $\delta_{-}\left(  t_{0},t\right)  =t$ for all $t\in\mathbb{T}^{\ast};$

\item If $\left(  s,t\right)  \in\mathcal{D}_{+},$ then $\delta_{+}\left(
s,t\right)  =u$ implies $\delta_{-}\left(  s,u\right)  =t$ and if $\left(
s,u\right)  \in\mathcal{D}_{-},$ then $\delta_{-}\left(  s,u\right)  =t$
implies $\delta_{+}\left(  s,t\right)  =u;$

\item $\delta_{+}\left(  t,\delta_{-}\left(  s,t_{0}\right)  \right)
=\delta_{-}\left(  s,t\right)  $ for all $\left(  s,t\right)  \in
\mathcal{D}_{+}$ with $t\geq t_{0};$

\item $\delta_{+}\left(  u,t\right)  =\delta_{+}\left(  t,u\right)  $ for all
$\left(  u,t\right)  \in\left(  \left[  t_{0},\infty\right)  _{\mathbb{T}%
}\times\left[  t_{0},\infty\right)  _{\mathbb{T}}\right)  \cap\mathcal{D}%
_{+};$

\item $\delta_{+}\left(  s,t\right)  \in\left[  t_{0},\infty\right)
_{\mathbb{T}}$ for all $\left(  s,t\right)  \in\mathcal{D}_{+}$ with $t\geq
t_{0};$

\item $\delta_{-}\left(  s,t\right)  \in\left[  t_{0},\infty\right)
_{\mathbb{T}}$ for all $\left(  s,t\right)  \in\left(  \left[  t_{0}%
,\infty\right)  _{\mathbb{T}}\times\left[  s,\infty\right)  _{\mathbb{T}%
}\right)  \cap\mathcal{D}_{-};$

\item If $\delta_{+}\left(  s,.\right)  $ is $\Delta$-differentiable in its
second variable, then $\delta_{+}^{\Delta_{t}}\left(  s,.\right)  >0;$

\item $\delta_{+}\left(  \delta_{-}\left(  u,s\right)  ,\delta_{-}\left(
s,v\right)  \right)  =$ $\delta_{-}\left(  u,v\right)  $ for all $\left(
s,v\right)  \in\left(  \left[  t_{0},\infty\right)  _{\mathbb{T}}\times\left[
s,\infty\right)  _{\mathbb{T}}\right)  \cap\mathcal{D}_{-}$ and $\left(
u,s\right)  \in\left(  \left[  t_{0},\infty\right)  _{\mathbb{T}}\times\left[
u,\infty\right)  _{\mathbb{T}}\right)  \cap\mathcal{D}_{-};$

\item If $\left(  s,t\right)  \in$ $\mathcal{D}_{-}$ and $\delta_{-}\left(
s,t\right)  =t_{0},$ then $s=t.$
\end{enumerate}

\begin{definition}
[Periodicity in shifts]\label{def3.2} Let $\mathbb{T}$ be a time scale with
the shift operators $\delta_{\pm}$ associated with the initial point $t_{0}%
\in\mathbb{T}^{\ast},$ then $\mathbb{T}$ is said to be periodic in shifts
$\delta_{\pm},$ if there exists a $p\in(t_{0},\infty)_{\mathbb{T}^{\ast}}$
such that $\left(  p,t\right)  \in\mathcal{D}_{\mp}$ for all $t\in
\mathbb{T}^{\ast}.$ $P$ is called the period of $\ \mathbb{T}$ if
\[
P=\inf\left\{  p\in(t_{0},\infty)_{\mathbb{T}^{\ast}}:\left(  p,t\right)
\in\mathcal{D}_{\mp}\text{ for all }t\in\mathbb{T}^{\ast}\right\}  >t_{0}.
\]

\end{definition}

Observe that an additive periodic time scale must be unbounded. The following
example indicates that a time scale, periodic in shifts, may be bounded.

\begin{example}
\label{ex new per}The following time scales are not additive periodic but
periodic in shifts $\delta_{\pm}$.

\begin{enumerate}
\item $\mathbb{T}_{1}\mathbb{=}\left\{  \pm n^{2}:n\in\mathbb{Z}\right\}  $,
$\delta_{\pm}(P,t)=\left\{
\begin{array}
[c]{ll}%
\left(  \sqrt{t}\pm\sqrt{P}\right)  ^{2} & \text{if }t>0\\
\pm P & \text{if }t=0\\
-\left(  \sqrt{-t}\pm\sqrt{P}\right)  ^{2} & \text{if }t<0
\end{array}
\right.  $, $P=1$, $t_{0}=0,$

\item $\mathbb{T}_{2}\mathbb{=}\overline{q^{\mathbb{Z}}}$, $\delta_{\pm
}(P,t)=P^{\pm1}t$, $P=q$, $t_{0}=1,$

\item $\mathbb{T}_{3}\mathbb{=}\overline{\mathbb{\cup}_{n\in\mathbb{Z}}\left[
2^{2n},2^{2n+1}\right]  }$, $\delta_{\pm}(P,t)=P^{\pm1}t$, $P=4$, $t_{0}=1,$

\item $\mathbb{T}_{4}\mathbb{=}\left\{  \frac{q^{n}}{1+q^{n}}:q>1\text{ is
constant and }n\in\mathbb{Z}\right\}  \cup\left\{  0,1\right\}  $,
\[
\delta_{\pm}(P,t)=\dfrac{q^{^{\left(  \frac{\ln\left(  \frac{t}{1-t}\right)
\pm\ln\left(  \frac{P}{1-P}\right)  }{\ln q}\right)  }}}{1+q^{\left(
\frac{\ln\left(  \frac{t}{1-t}\right)  \pm\ln\left(  \frac{P}{1-P}\right)
}{\ln q}\right)  }},\ \ P=\frac{q}{1+q},t_{0}=\frac{1}{2}.
\]

\end{enumerate}
\end{example}

Note that the time scale $\mathbb{T}_{4}$ in Example \ref{ex new per} is
bounded above and below and
\[
\mathbb{T}_{4}^{\ast}=\left\{  \frac{q^{n}}{1+q^{n}}:q>1\text{ is constant and
}n\in\mathbb{Z}\right\}  .
\]

\begin{corollary}
\label{Cor 1} Let $\mathbb{T}$ be a time scale that is periodic in shifts
$\delta_{\pm}$ with the period $P$. Then we have%
\begin{equation}
\delta_{\pm}(P,\sigma(t))=\sigma(\delta_{\pm}(P,t))\text{ for all }%
t\in\mathbb{T}^{\ast}\text{.} \label{sigma delta1}%
\end{equation}

\end{corollary}

\begin{example}
The time scale $\widetilde{\mathbb{T}}=(-\infty,0]\cup\lbrack1,\infty)$ cannot
be periodic in shifts $\delta_{\pm}$. Because if there was a $p\in
(t_{0},\infty)_{\widetilde{\mathbb{T}}^{\ast}}$ such that $\delta_{\pm
}(p,t)\in\widetilde{\mathbb{T}}^{\ast}$, then the point $\delta_{-}(p,0)$
would be right scattered due to (\ref{sigma delta1}). However, we have
$\delta_{-}(p,0)<0$ by (i) of Definition \ref{def3.1}. This leads to a
contradiction since every point less than $0$ is right dense.
\end{example}

\begin{definition}
[Periodic function in shifts $\delta_{\pm}$]\label{def3.3} Let $\mathbb{T}$ be
a time scale that is $P$-periodic in shifts $\delta_{\pm}$. We say that a real
valued function $f$ defined on $\mathbb{T}^{\ast}$ is periodic in shifts
$\delta_{\pm}$ if there exists a $T\in\left[  P,\infty\right)  _{\mathbb{T}%
^{\ast}}$ such that%
\begin{equation}
\left(  T,t\right)  \in\mathcal{D}_{\pm}\text{ and }f\left(  \delta_{\pm}%
^{T}\left(  t\right)  \right)  =f\left(  t\right)  \text{ for all }%
t\in\mathbb{T}^{\ast}, \label{3.1}%
\end{equation}
where $\delta_{\pm}^{T}\left(  t\right)  =\delta_{\pm}\left(  T,t\right)  $.
The number $T$ is called the period of $f,$ if it is the smallest number
satisfying (\ref{3.1}).
\end{definition}

\begin{example}
Let $\mathbb{T=R}$ with initial point $t_{0}=1,$ the function%
\[
f\left(  t\right)  =\sin\left(  \frac{\ln\left\vert t\right\vert }{\ln\left(
1/2\right)  }\pi\right)  ,\text{ }t\in\mathbb{R}^{\ast}:=\mathbb{R-}\left\{
0\right\}
\]
is four-periodic in shifts $\delta_{\pm}$ since%
\begin{align*}
f\left(  \delta_{\pm}\left(  4,t\right)  \right)   &  =\left\{
\begin{array}
[c]{c}%
f\left(  t4^{\pm1}\right)  \text{ if }t\geq0\\
f\left(  t/4^{\pm1}\right)  \text{ if }t<0
\end{array}
\right. \\
&  =\sin\left(  \frac{\ln\left\vert t\right\vert \pm2\ln\left(  1/2\right)
}{\ln\left(  1/2\right)  }\pi\right) \\
&  =\sin\left(  \frac{\ln\left\vert t\right\vert }{\ln\left(  1/2\right)  }%
\pi\pm2\pi\right) \\
&  =\sin\left(  \frac{\ln\left\vert t\right\vert }{\ln\left(  1/2\right)  }%
\pi\right) \\
&  =f\left(  t\right)  .
\end{align*}

\end{example}

\begin{definition}
[$\Delta$-periodic function in shifts $\delta_{\pm}$]\label{def3.4} Let
$\mathbb{T}$ be a time scale $P$-periodic in shifts. A real valued function
$f$ defined on $\mathbb{T}^{\ast}$ is $\Delta$-periodic function in shifts if
there exists a $T\in\left[  P,\infty\right)  _{\mathbb{T}^{\ast}}$ such that%
\begin{equation}
\left(  T,t\right)  \in\mathcal{D}_{\pm}\text{ for all }t\in\mathbb{T}^{\ast}
\label{3.2}%
\end{equation}%
\begin{equation}
\text{the shifts }\delta_{\pm}^{T}\text{ are }\Delta\text{-differentiable with
rd-continuous derivatives} \label{3.3}%
\end{equation}
and%
\begin{equation}
f\left(  \delta_{\pm}^{T}\left(  t\right)  \right)  \delta_{\pm}^{\Delta
T}\left(  t\right)  =f\left(  t\right)  \label{3.4}%
\end{equation}
for all $t\in\mathbb{T}^{\ast},$ where $\delta_{\pm}^{T}\left(  t\right)
=\delta_{\pm}\left(  T,t\right)  $. The smallest number $T$ satisfying
(\ref{3.2}-\ref{3.4}) is called period of $f$.

\begin{example}
The function $f\left(  t\right)  =1/t$ is $\Delta$-periodic function on
$q^{\mathbb{Z}}$ with the period $T=q$.
\end{example}
\end{definition}

The following result is useful for integration of functions which are $\Delta
$-periodic in shifts.

\begin{theorem}
Let $\mathbb{T}$ be a time scale that is periodic in shifts $\delta_{\pm}$
with period $P\in(t_{0},\infty)_{\mathbb{T}^{\ast}}$ and $f$ a $\Delta
$-periodic function in shifts $\delta_{\pm}$ with the period $T\in\left[
P,\infty\right)  _{\mathbb{T}^{\ast}}.$ Suppose that $f\in C_{rd}%
(\mathbb{T}),$ then%
\[%
{\displaystyle\int\limits_{t_{0}}^{t}}
f(s)\Delta s=%
{\displaystyle\int\limits_{\delta_{\pm}^{T}(t_{0})}^{\delta_{\pm}^{T}(t)}}
f(s)\Delta s.
\]

\end{theorem}

For more examples of periodic time scales, periodic functions and $\Delta
$-periodic functions in shifts, we may direct readers to \cite{[8]}.

\section{Floquet theory based on new periodicity concept}

In this section we use Lyapunov transformation and construct a unified Floquet
theory based on new periodicity concept to give necessary and sufficient
conditions for existence of periodic solutions of homogenous and
nonhomogeneous dynamic equations on time scales.

Hereafter, we suppose that $\mathbb{T}$ is a periodic time scale in shifts
$\delta_{\pm}$ and that the shift operators $\delta_{\pm}$ are $\Delta
$-differentiable with $rd$-continuous derivatives. For brevity, we use the
term "periodic in shifts" to mean periodicity in shifts $\delta_{\pm}$.
Throughout the paper, we use the notation $\delta_{\pm}^{T}\left(  t\right)  $
to indicate the shifts $\delta_{\pm}\left(  T,t\right)  $. Furthermore, we
denote by $\delta_{\pm}^{\left(  k\right)  }\left(  T,t\right)  $,
$k\in\mathbb{N}$, the $k$-times composition of shifts of $\delta_{\pm}^{T}$
with itself, namely,%
\[
\delta_{\pm}^{\left(  k\right)  }\left(  T,t\right)  :=\underset
{k-times}{\underbrace{\delta_{\pm}^{T}\circ\delta_{\pm}^{T}\circ...\circ
\delta_{\pm}^{T}}}\left(  t\right)  .
\]
Observe that, the period of a function $f$ does not have to be equal to period
of the time scale on which $f$ is determined. However, for simplicity of our
results we set the period of time scale $\mathbb{T}$ to be equal to period of
the all functions defined on $\mathbb{T}$.

\begin{definition}
\label{def4.1}\cite[Definition 2.1]{[3]}A Lyapunov transformation is an
invertible matrix $L\left(  t\right)  \in C_{rd}^{1}\left(  \mathbb{T}%
,\mathbb{R}^{n\times n}\right)  $ satisfying%
\[
\left\Vert L\left(  t\right)  \right\Vert \leq\rho\text{ and }\left\vert \det
L\left(  t\right)  \right\vert \geq\eta\text{ for all }t\in\mathbb{T}%
\]
where $\rho$ and $\eta$ are arbitrary positive reals.
\end{definition}

\subsection{Homogenous Case}

In this section we consider the regressive time varying linear dynamic initial
value problem%
\begin{equation}
x^{\Delta}\left(  t\right)  =A\left(  t\right)  x\left(  t\right)  ,\text{
}x\left(  t_{0}\right)  =x_{0}, \label{4.1}%
\end{equation}
where $A:\mathbb{T}^{\ast}\mathbb{\rightarrow R}^{n\times n}$ is $\Delta
$-periodic in shifts with period $T$. Note that if the time scale is
additive periodic, then $\delta_{\pm}^{\Delta}\left(  T,t\right)  =1$ and
$\Delta$-periodicity in shifts becomes the same as the periodicity in shifts.
Hence, the homogeneous system we consider in this section is more general than
that of \cite{[3]} and \cite{[4]}.

In \cite{[17]}, the solution of the system (\ref{4.1}) (for an arbitrary
matrix $A$) is expressed by the equality
\[
x\left(  t\right)  =\Phi_{A}\left(  t,t_{0}\right)  x_{0}\text{,}%
\]
where $\Phi_{A}\left(  t,t_{0}\right)  $, called the transition matrix for the
system (\ref{4.1}), is given by%
\begin{align}
\Phi_{A}\left(  t,t_{0}\right)   &  =I+%
{\displaystyle\int\limits_{t_{0}}^{t}}
A\left(  \tau_{1}\right)  \Delta\tau_{1}+%
{\displaystyle\int\limits_{t_{0}}^{t}}
A\left(  \tau_{1}\right)
{\displaystyle\int\limits_{t_{0}}^{\tau_{1}}}
A\left(  \tau_{2}\right)  \Delta\tau_{2}\Delta\tau_{1}+\ldots\nonumber\\
&  +%
{\displaystyle\int\limits_{t_{0}}^{t}}
A\left(  \tau_{1}\right)
{\displaystyle\int\limits_{t_{0}}^{\tau_{1}}}
A\left(  \tau_{2}\right)  \ldots%
{\displaystyle\int\limits_{t_{0}}^{\tau_{i-1}}}
A\left(  \tau_{i}\right)  \Delta\tau_{i}\ldots\Delta\tau_{1}+\ldots\text{.}
\label{re4.3}%
\end{align}
As mentioned in \cite{[3]} the matrix exponential $e_{A}\left(  t,t_{0}%
\right)  $ is not always identical to $\Phi_{A}\left(  t,t_{0}\right)  $ since%
\[
A\left(  t\right)  e_{A}\left(  t,t_{0}\right)  =e_{A}\left(  t,t_{0}\right)
A\left(  t\right)
\]
is always true but the equality
\[
A\left(  t\right)  \Phi_{A}\left(  t,t_{0}\right)  =\Phi_{A}\left(
t,t_{0}\right)  A\left(  t\right)
\]
is not. It can be seen from (\ref{4.3}) that one has $e_{A}\left(
t,t_{0}\right)  \equiv\Phi_{A}\left(  t,t_{0}\right)  $ only if the matrix $A$
satisfies%
\[
A\left(  t\right)
{\displaystyle\int\limits_{s}^{t}}
A\left(  \tau\right)  \Delta\tau=%
{\displaystyle\int\limits_{s}^{t}}
A\left(  \tau\right)  \Delta\tau A\left(  t\right)  .
\]

In preparation for the next result we define the set
\begin{equation}
P\left(  t_{0}\right)  :=\left\{  \delta_{+}^{\left(  k\right)  }\left(
T,t_{0}\right)  ,\text{ }k=0,1,2,\ldots\right\}  \label{P(t)}%
\end{equation}
and the function
\begin{equation}
\Theta\left(  t\right)  :=%
{\displaystyle\sum\limits_{j=1}^{m\left(  t\right)  }}
\delta_{-}\left(  \delta_{+}^{\left(  j-1\right)  }\left(  T,t_{0}\right)
,\delta_{+}^{\left(  j\right)  }\left(  T,t_{0}\right)  \right)  +G\left(
t\right)  , \label{4.1.1}%
\end{equation}
where
\begin{equation}
m\left(  t\right)  :=\min\left\{  k\in\mathbb{N}:\delta_{+}^{\left(  k\right)
}\left(  T,t_{0}\right)  \geq t\right\}  \label{m(t)}%
\end{equation}
and%
\begin{equation}
G\left(  t\right)  :=\left\{
\begin{array}
[c]{ll}%
0 & \text{if }t\in P\left(  t_{0}\right) \\
-\delta_{-}\left(  t,\delta_{+}^{\left(  m(t)\right)  }\left(  T,t_{0}\right)
\right)  & \text{if }t\notin P\left(  t_{0}\right)
\end{array}
\right.  . \label{G(t)}%
\end{equation}

\begin{remark}
For an additive periodic time scale we always have $\Theta\left(  t\right)
=t-t_{0}$.
\end{remark}

For the construction of matrix $R$, a solution of the matrix exponential
equation, it is necessary to define the real power of a matrix.

\begin{definition}
[Real power of a matrix]\label{def real power}\cite[Definition A.5]{[3]}Given
an $n\times n$ nonsingular matrix $M$ with elementary divisors $\left\{
\left(  \lambda-\lambda_{i}\right)  ^{m_{i}}\right\}  _{i=1}^{k}$ and any
$r\in\mathbb{R}$, the real power of the matrix $M$ is given by%
\begin{equation}
M^{r}:=\sum_{i=1}^{k}P_{i}\left(  M\right)  \lambda_{i}^{r}\left[  \sum
_{j=0}^{m_{i}-1}\frac{\Gamma\left(  r+1\right)  }{j!\Gamma\left(
r-j+1\right)  }\left(  \frac{M-\lambda_{i}I}{\lambda_{i}}\right)  ^{j}\right]
, \label{real power of matrix}%
\end{equation}
where%
\[
P_{i}\left(  \lambda\right)  :=a_{i}\left(  \lambda\right)  b_{i}\left(
\lambda\right)  ,
\]%
\[
b_{i}\left(  \lambda\right)  :=\Pi_{\substack{j\neq i\\j=1}}^{k}\left(
\lambda-\lambda_{j}\right)  ,
\]%
\[
\frac{1}{p\left(  \lambda\right)  }=\sum_{i=1}^{k}\frac{a_{i}\left(
\lambda\right)  }{\left(  \lambda-\lambda_{i}\right)  ^{m_{i}}},
\]
and $p\left(  \lambda\right)  $ is the characteristic polynomial of $M$.
\end{definition}

It has been deduced by \cite[Proposition A.3]{[3]} that the set $\left\{
P_{i}\left(  M\right)  \right\}  _{i=1}^{k}$ is orthogonal. That is, for any
$r,s\in\mathbb{R}$ we have $M^{s+r}=M^{s}M^{r}$.

In the following theorem we construct the matrix $R$ as a solution of matrix
exponential equation.

\begin{theorem}
\label{thm4.1} For a nonsingular, $n\times n$ constant matrix $M$ a solution
$R:\mathbb{T\rightarrow C}^{n\times n}$ of matrix exponential equation
\[
e_{R}\left(  \delta_{+}^{T}\left(  t_{0}\right)  ,t_{0}\right)  =M
\]
can be given by%
\begin{equation}
R\left(  t\right)  =\lim_{s\rightarrow t}\frac{M^{\frac{1}{T}\left[
\Theta\left(  \sigma\left(  t\right)  \right)  -\Theta\left(  s\right)
\right]  }-I}{\sigma\left(  t\right)  -s}, \label{4.1.2}%
\end{equation}
where $I$ is the $n\times n$ identity matrix and $\Theta$ is as in
(\ref{4.1.1}).
\end{theorem}

\begin{proof}
Let's construct the matrix exponential function $e_{R}\left(  t,t_{0}\right)
$ as follows%
\begin{equation}
e_{R}\left(  t,t_{0}\right)  :=M^{\frac{1}{T}\Theta\left(  t\right)  }\text{
for }t\geq t_{0}\text{,} \label{4.3}%
\end{equation}
where $\Theta$ is given by (\ref{4.1.1}) and real power of a nonsingular
matrix $M$ is given by (\ref{real power of matrix}). To show that the function
$e_{R}\left(  t,t_{0}\right)  $ constructed in (\ref{4.3}) is the matrix
exponential we first observe that%
\[
e_{R}\left(  t_{0},t_{0}\right)  =M^{\frac{1}{T}\Theta\left(  t_{0}\right)
}=I,
\]
where we use (\ref{4.3}) along with $\Theta\left(  t_{0}\right)  =G\left(
t_{0}\right)  =0$. Second, differentiating (\ref{4.3}) we obtain
\[
e_{R}^{\Delta}\left(  t,t_{0}\right)  =R\left(  t\right)  e_{R}\left(
t,t_{0}\right)  .
\]
To see this, first suppose that $t$ is right-scattered. Then, we have%
\begin{align*}
e_{R}^{\Delta}\left(  t,t_{0}\right)   &  =\frac{e_{R}\left(  \sigma\left(
t\right)  ,t_{0}\right)  -e_{R}\left(  t,t_{0}\right)  }{\sigma\left(
t\right)  -t}\\
&  =\frac{M^{\frac{1}{T}\Theta\left(  \sigma\left(  t\right)  \right)
}-M^{\frac{1}{T}\Theta\left(  t\right)  }}{\sigma\left(  t\right)  -t}\\
&  =\frac{M^{\frac{1}{T}[\Theta\left(  \sigma\left(  t\right)  \right)
-\Theta\left(  t\right)  ]}-I}{\sigma\left(  t\right)  -t}M^{\frac{1}{T}%
\Theta\left(  t\right)  }\\
&  =R\left(  t\right)  e_{R}\left(  t,t_{0}\right)  .
\end{align*}
If $t$ is right dense, then $\sigma\left(  t\right)  =t$. Setting $s=t+h$ in
(\ref{4.1.1}) and using (\ref{4.3}) we get%
\begin{align*}
e_{R}^{\Delta}\left(  t,t_{0}\right)   &  =\lim_{h\rightarrow0}\frac
{e_{R}\left(  t+h,t_{0}\right)  -e_{R}\left(  t,t_{0}\right)  }{h}\\
&  =\lim_{h\rightarrow0}\frac{M^{\frac{1}{T}\Theta\left(  t+h\right)
}-M^{\frac{1}{T}\Theta\left(  t\right)  }}{h}\\
&  =\lim_{h\rightarrow0}\frac{M^{\frac{1}{T}[\Theta\left(  t+h\right)
-\Theta\left(  t\right)  ]}-I}{h}M^{\frac{1}{T}\Theta\left(  t\right)  }\\
&  =R\left(  t\right)  e_{R}\left(  t,t_{0}\right)  .
\end{align*}
In any case, we have $e_{R}^{\Delta}\left(  t,t_{0}\right)  =R\left(
t\right)  e_{R}\left(  t,t_{0}\right)  $. Finally, it follows from Lemma
\ref{lem3.1} that%
\[
\Theta\left(  \delta_{+}^{T}\left(  t_{0}\right)  \right)  =\delta_{-}\left(
t_{0},\delta_{+}^{T}\left(  t_{0}\right)  \right)  =\delta_{+}^{T}\left(
t_{0}\right)  =T,
\]
and therefore,%
\[
e_{R}\left(  \delta_{+}^{T}\left(  t_{0}\right)  ,t_{0}\right)  =M^{\frac
{1}{T}\Theta\left(  \delta_{+}^{T}\left(  t_{0}\right)  \right)  }=M.
\]
The proof is complete.
\end{proof}

\begin{corollary}
\label{cor1}The matrices $R\left(  t\right)  $ and $M$ have identical eigenvectors.
\end{corollary}

\begin{proof}
For any eigenpairs $\{\lambda_{i},v_{i}\}$, $i=1,2,...,n$ of $M$, we get by
using $Mv_{i}=\lambda_{i}v_{i}$ that%
\[
\lim_{s\rightarrow t}M^{\frac{1}{T}[\Theta(\sigma(t))-\Theta(s)]}v_{i}%
=\lim_{s\rightarrow t}\lambda_{i}^{\frac{1}{T}[\Theta(\sigma(t))-\Theta
(s)]}v_{i}.
\]
This implies%
\begin{equation}
R(t)v_{i}=\lim_{s\rightarrow t}\left(  \frac{\lambda_{i}^{\frac{1}{T}%
[\Theta(\sigma(t))-\Theta(s)]}-1}{\sigma(t)-s}\right)  v_{i}. \label{2.0}%
\end{equation}
Substituting $\gamma_{i}(t)=\lim_{s\rightarrow t}\left(  \frac{\lambda
_{i}^{\frac{1}{T}[\Theta(\sigma(t))-\Theta(s)]}-1}{\sigma(t)-s}\right)  $ into
(\ref{2.0}) we conclude that $R(t)$ has the eigenpairs $\{\gamma_{i}%
(t),v_{i}\}_{i=1}^{n}$.
\end{proof}

\begin{lemma}
\label{lem4.2}Let $\mathbb{T}$ be a time scale and $P\in\mathcal{R}\left(
\mathbb{T}^{\ast},\mathbb{R}^{n\times n}\right)  $ be a $\Delta-$periodic
matrix valued function in shifts with period $T$, i.e.%
\[
P\left(  t\right)  =P\left(  \delta_{\pm}^{T}\left(  t\right)  \right)
\delta_{\pm}^{\Delta T}\left(  t\right)
\]
Then the solution of the dynamic matrix initial value problem%
\begin{equation}
Y^{\Delta}\left(  t\right)  =P\left(  t\right)  Y\left(  t\right)  ,\text{
}Y\left(  t_{0}\right)  =Y_{0}, \label{2}%
\end{equation}
is unique up to a period $T$ in shifts. That is
\begin{equation}
\Phi_{P}\left(  t,t_{0}\right)  =\Phi_{P}\left(  \delta_{+}^{T}\left(
t\right)  ,\delta_{+}^{T}\left(  t_{0}\right)  \right)  \label{2.1}%
\end{equation}
for all $t\in\mathbb{T}^{\ast}$.
\end{lemma}

\begin{proof}
By \cite[Theorem 3.2]{[17]}, the unique solution to (\ref{2}) is $Y\left(  t\right)
=\Phi_{P}\left(  t,t_{0}\right)  Y_{0}.$ Observe that%
\[
Y^{\Delta}\left(  t\right)  =\Phi_{P}^{\Delta}\left(  t,t_{0}\right)
Y_{0}=P\left(  t\right)  \Phi_{P}\left(  t,t_{0}\right)  Y_{0}%
\]
and%
\[
Y\left(  t_{0}\right)  =\Phi_{P}\left(  t_{0},t_{0}\right)  Y_{0}=Y_{0}.
\]
To verify (\ref{2.1}) we first need to show that $\Phi_{P}\left(  \delta
_{+}^{T}\left(  t\right)  ,\delta_{+}^{T}\left(  t_{0}\right)  \right)  Y_{0}$
is also solution for (\ref{2}). Since the shift operator $\delta_{+}$ is
strictly increasing, the chain rule (\cite[Theorem 1.93]{[1]}) yields%
\begin{align*}
\left[  \Phi_{P}\left(  \delta_{+}^{T}\left(  t\right)  ,\delta_{+}^{T}\left(
t_{0}\right)  \right)  Y_{0}\right]  ^{\Delta}  &  =P\left(  \delta_{\pm}%
^{T}\left(  t\right)  \right)  \delta_{\pm}^{\Delta T}\left(  t\right)
\Phi_{P}\left(  \delta_{+}^{T}\left(  t\right)  ,\delta_{+}^{T}\left(
t_{0}\right)  \right)  Y_{0}\\
&  =P\left(  t\right)  \Phi_{P}\left(  \delta_{+}^{T}\left(  t\right)
,\delta_{+}^{T}\left(  t_{0}\right)  \right)  Y_{0}.
\end{align*}
On the other hand, we have
\[
\Phi_{P}\left(  \delta_{+}^{T}\left(  t\right)  ,\delta_{+}^{T}\left(
t_{0}\right)  \right)  _{t=t_{0}}Y_{0}=\Phi_{P}\left(  \delta_{+}^{T}\left(
t_{0}\right)  ,\delta_{+}^{T}\left(  t_{0}\right)  \right)  Y_{0}=Y_{0}.
\]
This means $\Phi_{P}\left(  \delta_{+}^{T}\left(  t\right)  ,\delta_{+}%
^{T}\left(  t_{0}\right)  \right)  Y_{0}$ solves (\ref{2}). From the
uniqueness of solution of (\ref{2}), we get (\ref{2.1}).
\end{proof}

One may similarly prove the next result.

\begin{corollary}
Let $\mathbb{T}$ be a time scale and $P\in\mathcal{R}\left(  \mathbb{T}^{\ast
},\mathbb{R}^{n\times n}\right)  $ be a $\Delta-$periodic matrix valued
function in shifts, i.e.%
\[
P\left(  t\right)  =P\left(  \delta_{\pm}^{T}\left(  t\right)  \right)
\delta_{\pm}^{\Delta T}\left(  t\right)
\]
Then%
\begin{equation}
e_{P}\left(  t,t_{0}\right)  =e_{P}\left(  \delta_{+}^{T}\left(  t\right)
,\delta_{+}^{T}\left(  t_{0}\right)  \right)  . \label{2.2}%
\end{equation}

\end{corollary}

\begin{theorem}
[Floquet decomposition]\label{thm1} Let $A$ be a matrix valued function that
is $\Delta$-periodic in shifts with period $T$. The transition matrix for $A$
can be given in the form%
\begin{equation}
\Phi_{A}\left(  t,\tau\right)  =L\left(  t\right)  e_{R}\left(  t,\tau\right)
L^{-1}\left(  \tau\right)  ,\text{ for all }t,\tau\in\mathbb{T}^{\ast},
\label{3}%
\end{equation}
where $R:\mathbb{T\rightarrow C}^{n\times n}$ is $\Delta$-periodic function in shifts and $L\left(  t\right)  \in
C_{rd}^{1}\left(  \mathbb{T}^{\ast},\mathbb{R}^{n\times n}\right)  $ is
periodic in shifts with the same period $T$.
\end{theorem}

\begin{proof}
Setting $M:=\Phi_{A}\left(  \delta_{+}^{T}\left(  t_{0}\right)  ,t_{0}\right)
$ define the matrix $R$ as in Theorem \ref{thm4.1}. Then we have%
\[
e_{R}\left(  \delta_{+}^{T}\left(  t_{0}\right)  ,t_{0}\right)  =\Phi
_{A}\left(  \delta_{+}^{T}\left(  t_{0}\right)  ,t_{0}\right)  .
\]
Define the matrix $L\left(  t\right)  $ by%
\begin{equation}
L\left(  t\right)  :=\Phi_{A}\left(  t,t_{0}\right)  e_{R}^{-1}\left(
t,t_{0}\right)  . \label{3.6}%
\end{equation}
Obviously, $L\left(  t\right)  \in C_{rd}^{1}\left(  \mathbb{T}^{\ast
},\mathbb{R}^{n\times n}\right)  $ and $L$ is invertible. The equality%
\begin{equation}
\Phi_{A}\left(  t,t_{0}\right)  =L\left(  t\right)  e_{R}\left(
t,t_{0}\right)  . \label{4}%
\end{equation}
along with (\ref{4}) implies%
\begin{align}
\Phi_{A}\left(  t_{0},t\right)   &  =e_{R}^{-1}\left(  t,t_{0}\right)
L^{-1}\left(  t\right) \nonumber\\
&  =e_{R}\left(  t_{0},t\right)  L^{-1}\left(  t\right)  . \label{5}%
\end{align}
Combining (\ref{4}) and (\ref{5}), we obtain (\ref{3}). To show periodicity of
$L$ in shifts we use (\ref{2.1}-\ref{2.2}) to get%
\begin{align*}
L\left(  \delta_{+}^{T}\left(  t\right)  \right)   &  =\Phi_{A}\left(
\delta_{+}^{T}\left(  t\right)  ,t_{0}\right)  e_{R}^{-1}\left(  \delta
_{+}^{T}\left(  t\right)  ,t_{0}\right) \\
&  =\Phi_{A}\left(  \delta_{+}^{T}\left(  t\right)  ,\delta_{+}^{T}\left(
t_{0}\right)  \right)  \Phi_{A}\left(  \delta_{+}^{T}\left(  t_{0}\right)
,t_{0}\right)  e_{R}\left(  t_{0},\delta_{+}^{T}\left(  t\right)  \right) \\
&  =\Phi_{A}\left(  \delta_{+}^{T}\left(  t\right)  ,\delta_{+}^{T}\left(
t_{0}\right)  \right)  \Phi_{A}\left(  \delta_{+}^{T}\left(  t_{0}\right)
,t_{0}\right)  e_{R}\left(  t_{0},\delta_{+}^{T}\left(  t_{0}\right)  \right)
e_{R}\left(  \delta_{+}^{T}\left(  t_{0}\right)  ,\delta_{+}^{T}\left(
t\right)  \right) \\
&  =\Phi_{A}\left(  \delta_{+}^{T}\left(  t\right)  ,\delta_{+}^{T}\left(
t_{0}\right)  \right)  e_{R}\left(  \delta_{+}^{T}\left(  t_{0}\right)
,\delta_{+}^{T}\left(  t\right)  \right) \\
&  =\Phi_{A}\left(  \delta_{+}^{T}\left(  t\right)  ,\delta_{+}^{T}\left(
t_{0}\right)  \right)  e_{R}^{-1}\left(  \delta_{+}^{T}\left(  t\right)
,\delta_{+}^{T}\left(  t_{0}\right)  \right) \\
&  =\Phi_{A}\left(  t,t_{0}\right)  e_{R}^{-1}\left(  t,t_{0}\right) \\
&  =L\left(  t\right)  .
\end{align*}
This completes the proof.
\end{proof}

Hereafter, we shall refer to (\ref{3}) as the \textit{Floquet decomposition}
for $\Phi_{A}$. The following result can be proven similar to \cite[Theorem
3.7]{[3]}.

\begin{theorem}
\label{thm2} Let $\Phi_{A}\left(  t,t_{0}\right)  =L\left(  t\right)
e_{R}\left(  t,t_{0}\right)  $ be a Floquet decomposition for $\Phi_{A}$.
Then, $x\left(  t\right)  =\Phi_{A}\left(  t,t_{0}\right)  x_{0}$ is a
solution of the $T$-periodic system (\ref{4.1}) if and only if \ $z\left(
t\right)  =L^{-1}\left(  t\right)  x\left(  t\right)  $ is a solution of the
system%
\[
z^{\Delta}\left(  t\right)  =R\left(  t\right)  z\left(  t\right)  ,\text{
}z\left(  t_{0}\right)  =x_{0}.
\]

\end{theorem}

\begin{theorem}
\label{thm3}There exists an initial state $x\left(  t_{0}\right)  =x_{0}\neq0$
such that the solution of (\ref{4.1}) is $T$-periodic in shifts if and only if
one of the eigenvalues of
\[
e_{R}\left(  \delta_{+}^{T}\left(  t_{0}\right)  ,t_{0}\right)  =\Phi
_{A}\left(  \delta_{+}^{T}\left(  t_{0}\right)  ,t_{0}\right)
\]
is $1$.
\end{theorem}

\begin{proof}
Suppose that $x\left(  t_{0}\right)  =x_{0}$ and $x\left(  t\right)  $ is a
solution of (\ref{4.1}) which is $T$-periodic in shifts. By Theorem
\ref{thm1}, the Floquet decomposition of $x$ is given by%
\[
x\left(  t\right)  =\Phi_{A}\left(  t,t_{0}\right)  x_{0}=L\left(  t\right)
e_{R}\left(  t,t_{0}\right)  L^{-1}\left(  t_{0}\right)  x_{0},
\]
which also yields%
\[
x\left(  \delta_{+}^{T}\left(  t\right)  \right)  =L\left(  \delta_{+}%
^{T}\left(  t\right)  \right)  e_{R}\left(  \delta_{+}^{T}\left(  t\right)
,t_{0}\right)  L^{-1}\left(  t_{0}\right)  x_{0}.
\]
By $T$-periodicity of $x$ and $L$ in shifts, we have%
\[
e_{R}\left(  t,t_{0}\right)  L^{-1}\left(  t_{0}\right)  x_{0}=e_{R}\left(
\delta_{+}^{T}\left(  t\right)  ,t_{0}\right)  L^{-1}\left(  t_{0}\right)
x_{0},
\]
and therefore,%
\[
e_{R}\left(  t,t_{0}\right)  L^{-1}\left(  t_{0}\right)  x_{0}=e_{R}\left(
\delta_{+}^{T}\left(  t\right)  ,\delta_{+}^{T}\left(  t_{0}\right)  \right)
e_{R}\left(  \delta_{+}^{T}\left(  t_{0}\right)  ,t_{0}\right)  L^{-1}\left(
t_{0}\right)  x_{0}.
\]
Since $e_{R}\left(  \delta_{+}^{T}\left(  t\right)  ,\delta_{+}^{T}\left(
t_{0}\right)  \right)  =e_{R}\left(  t,t_{0}\right)  $ the last equality
implies%
\[
e_{R}\left(  t,t_{0}\right)  L^{-1}\left(  t_{0}\right)  x_{0}=e_{R}\left(
t,t_{0}\right)  e_{R}\left(  \delta_{+}^{T}\left(  t_{0}\right)
,t_{0}\right)  L^{-1}\left(  t_{0}\right)  x_{0}%
\]
and thus%
\[
L^{-1}\left(  t_{0}\right)  x_{0}=e_{R}\left(  \delta_{+}^{T}\left(
t_{0}\right)  ,t_{0}\right)  L^{-1}\left(  t_{0}\right)  x_{0}.
\]
Since $L^{-1}\left(  t_{0}\right)  x_{0}\neq0$, we see that $L^{-1}\left(
t_{0}\right)  x_{0}$ is an eigenvector of the matrix $e_{R}\left(  \delta
_{+}^{T}\left(  t_{0}\right)  ,t_{0}\right)  $ corresponding to an eigenvalue
of $1.$

Conversely, let us assume that $1$ is an eigenvalue of $e_{R}\left(
\delta_{+}^{T}\left(  t_{0}\right)  ,t_{0}\right)  $ with corresponding
eigenvector $z_{0.}$ This means $z_{0}$ is real valued and nonzero. Using
$e_{R}\left(  t,t_{0}\right)  =e_{R}\left(  \delta_{+}^{T}\left(  t\right)
,\delta_{+}^{T}\left(  t_{0}\right)  \right)  $, we arrive at the following
equality%
\begin{align*}
z\left(  \delta_{+}^{T}\left(  t\right)  \right)   &  =e_{R}\left(  \delta
_{+}^{T}\left(  t\right)  ,t_{0}\right)  z_{0}\\
&  =e_{R}\left(  \delta_{+}^{T}\left(  t\right)  ,\delta_{+}^{T}\left(
t_{0}\right)  \right)  e_{R}\left(  \delta_{+}^{T}\left(  t_{0}\right)
,t_{0}\right)  z_{0}\\
&  =e_{R}\left(  \delta_{+}^{T}\left(  t\right)  ,\delta_{+}^{T}\left(
t_{0}\right)  \right)  z_{0}\\
&  =e_{R}\left(  t,t_{0}\right)  z_{0}\\
&  =z\left(  t\right)  ,
\end{align*}
which shows that $z\left(  t\right)  =e_{R}\left(  t,t_{0}\right)z_{0} $ is
$T$-periodic in shifts. Applying the Floquet decomposition and setting
$x_{0}:=L\left(  t_{0}\right)  z_{0}$, we obtain the nontrivial solution $x$
of (\ref{4.1}) as follows%
\[
x\left(  t\right)  =\Phi_{A}\left(  t,t_{0}\right)  x_{0}=L\left(  t\right)
e_{R}\left(  t,t_{0}\right)  L^{-1}\left(  t_{0}\right)  x_{0}=L\left(
t\right)  e_{R}\left(  t,t_{0}\right)  z_{0}=L\left(  t\right)  z\left(
t\right)  ,
\]
which is $T$-periodic in shifts since $L$ and $z$ are $T$-periodic in shifts.
\end{proof}

\subsection{Nonhomogeneous Case}

Let us focus on the nonhomogeneous regressive time varying linear dynamic
initial value problem%
\begin{equation}
x^{\Delta}\left(  t\right)  =A\left(  t\right)  x\left(  t\right)  +F\left(
t\right)  ,\text{ }x\left(  t_{0}\right)  =x_{0}, \label{6}%
\end{equation}
where $A:\mathbb{T}^{\ast}\mathbb{\rightarrow R}^{n\times n}$, $F\in
C_{rd}\left(  \mathbb{T}^{\ast},\mathbb{R}^{n}\right)  \cap\mathcal{R}\left(
\mathbb{T}^{\ast},\mathbb{R}^{n}\right)  $. Hereafter, we suppose both $A$ and
$F$ are $\Delta$-periodic in shifts with the period $T$.

\begin{lemma}
\label{lem4}A solution $x\left(  t\right)  $ of (\ref{6}) is $T$-periodic in
shifts if and only if $x\left(  \delta_{+}^{T}\left(  t\right)  \right)
=x\left(  t\right)  $ for all $t\in\mathbb{T}^{\ast}$.
\end{lemma}

\begin{proof}
Suppose that $x\left(  t\right)  $ is $T$-periodic in shifts. Let us define
$z\left(  t\right)  $ as
\begin{equation}
z\left(  t\right)  =x\left(  \delta_{+}^{T}\left(  t\right)  \right)
-x\left(  t\right)  . \label{7}%
\end{equation}
Obviously $z\left(  t_{0}\right)  =0.$ Moreover, if we take delta derivative
of both sides of (\ref{7}), we have the following:%
\begin{align*}
z^{\Delta}\left(  t\right)   &  =\left[  x\left(  \delta_{+}^{T}\left(
t\right)  \right)  -x\left(  t\right)  \right]  ^{\Delta}\\
&  =x^{\Delta}\left(  \delta_{+}^{T}\left(  t\right)  \right)  -x^{\Delta
}\left(  t\right) \\
&  =x^{\Delta}\left(  \delta_{+}^{T}\left(  t\right)  \right)  \delta
_{+}^{^{\Delta}T}\left(  t\right)  -x^{\Delta}\left(  t\right) \\
&  =A\left(  \delta_{+}^{T}\left(  t\right)  \right)  x\left(  \delta_{+}%
^{T}\left(  t\right)  \right)  \delta_{+}^{^{\Delta}T}\left(  t\right)
+F\left(  \delta_{+}^{T}\left(  t\right)  \right)  \delta_{+}^{^{\Delta}%
T}\left(  t\right)  -A\left(  t\right)  x\left(  t\right)  -F\left(  t\right)
.
\end{align*}
Since $A$ and $F$ are both $\Delta$-periodic in shifts with the period $T$, we
have%
\begin{align*}
z^{\Delta}\left(  t\right)   &  =A\left(  t\right)  x\left(  \delta_{+}%
^{T}\left(  t\right)  \right)  +F\left(  t\right)  -A\left(  t\right)
x\left(  t\right)  -F\left(  t\right) \\
&  =A\left(  t\right)  \left[  x\left(  \delta_{+}^{T}\left(  t\right)
\right)  -x\left(  t\right)  \right] \\
&  =A\left(  t\right)  z\left(  t\right)  .
\end{align*}
By uniqueness of solutions, we can conclude that $z\left(  t\right)  \equiv0$
and that $x\left(  \delta_{+}^{T}\left(  t\right)  \right)  =x\left(
t\right)  $ for all $t\in\mathbb{T}^{\ast}.$
\end{proof}

\begin{theorem}
\label{thm5} For any initial point $t_{0}\in\mathbb{T}^{\ast}$ and for any
function $F$ that is $\Delta$-periodic in shifts with period $T$, there exists
an initial state $x\left(  t_{0}\right)  =x_{0}$ such that the solution of
(\ref{6}) is $T$-periodic in shifts if and only if there does not exist a
nonzero $z\left(  t_{0}\right)  =z_{0}$ and $t_{0}\in\mathbb{T}^{\ast}$ such
that the homogeneous initial value problem%
\begin{equation}
z^{\Delta}\left(  t\right)  =A\left(  t\right)  z\left(  t\right)  ,\text{
}z\left(  t_{0}\right)  =z_{0}, \label{8}%
\end{equation}
has a solution that is $T$-periodic in shifts.
\end{theorem}

\begin{proof}
In \cite{[16]}, the following representation for the solution of (\ref{6}) is
given%
\[
x\left(  t\right)  =X\left(  t\right)  X^{-1}\left(  \tau\right)  x_{0}+%
{\displaystyle\int\limits_{\tau}^{t}}
X\left(  t\right)  X^{-1}\left(  \sigma\left(  s\right)  \right)  F\left(
s\right)  \Delta s,
\]
where $X\left(  t\right)  $ is a fundamental matrix solution of the homogenous
system (\ref{4.1}) with respect to initial condition $x\left(  \tau\right)
=x_{0}$. As it is done in \cite{[16]}, we can express $x\left(  t\right)  $ as
follows%
\[
x\left(  t\right)  =\Phi_{A}\left(  t,t_{0}\right)  x_{0}+\int_{t_{0}}^{t}%
\Phi_{A}\left(  t,\sigma\left(  s\right)  \right)  F\left(  s\right)  \Delta
s.
\]

By the previous lemma we know that $x\left(  t\right)  $ is $T$-periodic in
shifts if and only if $x\left(  \delta_{+}^{T}\left(  t_{0}\right)  \right)
=x_{0}$ or equivalently
\begin{equation}
\left[  I-\Phi_{A}\left(  \delta_{+}^{T}\left(  t_{0}\right)  ,t_{0}\right)
\right]  x_{0}=%
{\displaystyle\int\limits_{t_{0}}^{\delta_{+}^{T}\left(  t_{0}\right)  }}
\Phi_{A}\left(  \delta_{+}^{T}\left(  t_{0}\right)  ,\sigma\left(  s\right)
\right)  F\left(  s\right)  \Delta s. \label{8**}%
\end{equation}
By guidance of Theorem \ref{thm3}, we have to show that (\ref{6}) has a
solution with respect to initial condition $x\left(  t_{0}\right)  =x_{0}$ if
and only if $e_{R}\left(  \delta_{+}^{T}\left(  t_{0}\right)  ,t_{0}\right)  $
has no eigenvalues equal to $1.$

Let $e_{R}\left(  \delta_{+}^{T}\left(  \eta\right)  ,\eta\right)  =\Phi
_{A}\left(  \delta_{+}^{T}\left(  \eta\right)  ,\eta\right)  $, for some
$\eta\in\mathbb{T}^{\ast}$, has no eigenvalues equal to $1$. That is,%
\[
\det\left[  I-\Phi_{A}\left(  \delta_{+}^{T}\left(  \eta\right)  ,\eta\right)
\right]  \neq0.
\]
Invertibility and periodicity of $\Phi_{A}$ imply%
\begin{align}
0  &  \neq\det\left[  \Phi_{A}\left(  \delta_{+}^{T}\left(  t_{0}\right)
,\delta_{+}^{T}\left(  \eta\right)  \right)  \left(  I-\Phi_{A}\left(
\delta_{+}^{T}\left(  \eta\right)  ,\eta\right)  \right)  \Phi_{A}\left(
\eta,t_{0}\right)  \right] \nonumber\\
&  =\det\left[  \Phi_{A}\left(  \delta_{+}^{T}\left(  t_{0}\right)
,\delta_{+}^{T}\left(  \eta\right)  \right)  \Phi_{A}\left(  \eta
,t_{0}\right)  -\Phi_{A}\left(  \delta_{+}^{T}\left(  t_{0}\right)
,t_{0}\right)  \right]  . \label{9}%
\end{align}
By periodicity of $\Phi_{A}$, the invertibility of $\left[  I-\Phi_{A}\left(
\delta_{+}^{T}\left(  t_{0}\right)  ,t_{0}\right)  \right]  $ is equivalent to
(\ref{9}) for any $t_{0}\in\mathbb{T}^{\ast}$. Thus, (\ref{8**}) has a
solution%
\[
x_{0}=\left[  I-\Phi_{A}\left(  \delta_{+}^{T}\left(  t_{0}\right)
,t_{0}\right)  \right]  ^{-1}%
{\displaystyle\int\limits_{t_{0}}^{\delta_{+}^{T}\left(  t_{0}\right)  }}
\Phi_{A}\left(  \delta_{+}^{T}\left(  t_{0}\right)  ,\sigma\left(  s\right)
\right)  F\left(  s\right)  \Delta s
\]
for any $t_{0}\in\mathbb{T}^{\ast}$ and for any $\Delta$-periodic function $F$
in shifts with period $T$.

Suppose that (\ref{8**}) has a solution for every $t_{0}\in\mathbb{T}^{\ast}$
and every $\Delta$-periodic function $F$ in shifts with period $T$. Let us
define the set $P_{-}\left(  t\right)  $ as%
\[
P_{-}\left(  t\right)  =\left\{  k\in\mathbb{Z}:\delta_{-}^{\left(  k\right)
}\left(  T,t\right)  \right\}  .
\]
It is clear that, $P_{-}\left(  t\right)  =P_{-}\left(  \delta_{+}^{T}\left(
t\right)  \right)  $. Additionally, let the function $\xi$ be defined by%
\begin{align*}
\xi\left(  t\right)   &  :=%
{\displaystyle\prod\limits_{s\in P_{-}\left(  t\right)  \cap\left[
t_{0},t\right)  }}
\left(  \delta_{+}^{\Delta T}\left(  s\right)  \right)  ^{-1}\\
&  =\left(  \delta_{+}^{\Delta T}\left(  \delta_{-}\left(  T,t\right)
\right)  \right)  ^{-1}\times\left(  \delta_{+}^{\Delta T}\left(  \delta
_{-}^{\left(  2\right)  }\left(  T,t\right)  \right)  \right)  ^{-1}%
\times\ldots\times\left(  \delta_{+}^{\Delta T}\left(  \delta_{-}^{\left(
m^{-}\left(  t\right)  \right)  }\left(  T,t\right)  \right)  \right)  ^{-1},
\end{align*}
where $m^{-}\left(  t\right)  =\max\left\{  k\in\mathbb{Z}:\delta_{-}%
^{(k)}\left(  T,t\right)  \geq t_{0}\right\}  $. By definition of $\xi$, we
have%
\begin{align*}
\xi\left(  \delta_{+}^{T}\left(  t\right)  \right)   &  =%
{\displaystyle\prod\limits_{s\in P_{-}\left(  \delta_{+}^{T}\left(
t\right)  \right)  \cap\left[  t_{0},\delta_{+}^{T}\left(  t\right)
\right)  }}
\left(  \delta_{+}^{\Delta T}(s)\right)  ^{-1}\\
&  =%
{\displaystyle\prod\limits_{s\in P_{-}\left(  t\right)  \cap\left[
t_{0},\delta_{+}^{T}\left(  t\right)  \right)  }}
\left(  \delta_{+}^{\Delta T}\left(  s\right)  \right)  ^{-1}\\
&  =\left(  \delta_{+}^{\Delta T}\left(  t\right)  \right)  ^{-1}%
{\displaystyle\prod\limits_{s\in P_{-}\left(  t\right)  \cap\left[
t_{0},t\right)  }}
\left(  \delta_{+}^{\Delta T}\left(  s\right)  \right)  ^{-1}\\
&  =\left(  \delta_{+}^{\Delta T}\left(  t\right)  \right)  ^{-1}\xi\left(
t\right)  ,
\end{align*}
which shows that $\xi$ is $\Delta$-periodic in shifts with period $T$. For an
arbitrary $t_{0}$ and corresponding $F_{0},$ we can define a regressive and
$\Delta$-periodic function $F$ in shifts as follows%
\begin{equation}
F\left(  t\right)  :=\Phi_{A}\left(  \sigma\left(  t\right)  ,\delta_{+}%
^{T}\left(  t_{0}\right)  \right)  \xi\left(  t\right)  F_{0},\ \ t\in\left[
t_{0},\delta_{+}^{T}\left(  t_{0}\right)  \right)  \cap\mathbb{T}. \label{9*}%
\end{equation}
Then, we have%
\begin{equation}%
{\displaystyle\int\limits_{t_{0}}^{\delta_{+}^{T}\left(  t_{0}\right)  }}
\Phi_{A}\left(  \delta_{+}^{T}\left(  t_{0}\right)  ,\sigma\left(  s\right)
\right)  F\left(  s\right)  \Delta s=F_{0}%
{\displaystyle\int\limits_{t_{0}}^{\delta_{+}^{T}\left(  t_{0}\right)  }}
\xi\left(  s\right)  \Delta s. \label{9**}%
\end{equation}
Thus, (\ref{8**}) can be rewritten as follows%
\begin{equation}
\left[  I-\Phi_{A}\left(  \delta_{+}^{T}\left(  t_{0}\right)  ,t_{0}\right)
\right]  x_{0}=%
{\displaystyle\int\limits_{t_{0}}^{\delta_{+}^{T}\left(  t_{0}\right)  }}
\xi\left(  s\right)  \Delta s. \label{9***}%
\end{equation}
For any $F$ that is constructed in (\ref{9*}), and hence for any corresponding
$F_{0}$, (\ref{9***}) has a solution for $x_{0}$ by assumption. Therefore,%
\[
\det\left[  I-\Phi_{A}\left(  \delta_{+}^{T}\left(  t_{0}\right)
,t_{0}\right)  \right]  \neq0.
\]
Consequently, $e_{R}\left(  \delta_{+}^{T}\left(  t_{0}\right)  ,t_{0}\right)
=\Phi_{A}\left(  \delta_{+}^{T}\left(  t_{0}\right)  ,t_{0}\right)  $ has no
eigenvalue $1$. Then, we can conclude by Theorem \ref{thm3}, (\ref{8}) has no
periodic solution in shifts. The proof is complete.
\end{proof}

\begin{example}
\label{examp1}Consider the time scale $\mathbb{T}=\overline{q^{\mathbb{Z}}}$
that is $q$- periodic in shifts $\delta_{\pm}\left(  s,t\right)  =s^{\pm1}t$
associated with the initial point $t_{0}=1$. Let us define the matrix function
$A\left(  t\right)  :\mathbb{T}^{\ast}\mathbb{\rightarrow R}^{n\times n}$ as
follows%
\[
A\left(  t\right)  =\left[
\begin{array}
[c]{cc}%
\frac{1}{t} & 0\\
0 & \frac{1}{t}%
\end{array}
\right]  .
\]
Then%
\[
A\left(  \delta_{+}^{q}\left(  t\right)  \right)  \delta_{+}^{\Delta q}\left(
t\right)  =\left[
\begin{array}
[c]{cc}%
\frac{1}{qt} & 0\\
0 & \frac{1}{qt}%
\end{array}
\right]  \times q=\left[
\begin{array}
[c]{cc}%
\frac{1}{t} & 0\\
0 & \frac{1}{t}%
\end{array}
\right]  =A\left(  t\right)  ,
\]
which shows that $A$ is $\Delta$-periodic in shifts with period $q$. \newline
Consider the system%
\[
x^{\Delta}\left(  t\right)  =\left[
\begin{array}
[c]{cc}%
\frac{1}{t} & 0\\
0 & \frac{1}{t}%
\end{array}
\right]  x\left(  t\right)  ,
\]
with the transition matrix $\Phi_{A}\left(  t,1\right)  $ given by%
\[
\Phi_{A}\left(  t,1\right)  =\left[
\begin{array}
[c]{cc}%
e_{1/t}(t,1) & 0\\
0 & e_{1/t}(t,1)
\end{array}
\right]  ,
\]
where $q$-exponential function defined as%
\[
e_{p}(t,t_{0})=%
{\displaystyle\prod\limits_{s\in\lbrack t_{0},t)}}
\left[  1+(q-1)sp(s)\right]  .
\]
By (\ref{2.1}), we get%
\[
\Phi_{A}\left(  \delta_{+}^{q}\left(  t\right)  ,\delta_{+}^{q}\left(
1\right)  \right)  =\Phi_{A}\left(  t,1\right)
\]
and%
\[
\Phi_{A}\left(  \delta_{+}^{q}\left(  1\right)  ,1\right)  =\Phi_{A}\left(
q,1\right)  =\left[
\begin{array}
[c]{cc}%
q & 0\\
0 & q
\end{array}
\right]  .
\]
Now, as in Theorem \ref{thm4.1} we have%
\[
e_{R}\left(  q,1\right)  =\Phi_{A}\left(  q,1\right)  =\left[
\begin{array}
[c]{cc}%
q & 0\\
0 & q
\end{array}
\right]  =M.
\]
Then $R\left(  t\right)  $ in the Floquet decomposition is given by%
\begin{align*}
R\left(  t\right)   &  =\frac{1}{qt-t}\left[  M^{\frac{1}{q}\left(
\Theta\left(  qt\right)  -\Theta\left(  t\right)  \right)  }-I\right] \\
&  =\frac{1}{\left(  q-1\right)  t}\left[  M^{\frac{1}{q}\times q}-I\right] \\
&  =\frac{1}{\left(  q-1\right)  t}\left[  M-I\right] \\
&  =%
\begin{bmatrix}
\frac{q-1}{\left(  q-1\right)  t} & 0\\
0 & \frac{q-1}{\left(  q-1\right)  t}%
\end{bmatrix}
=%
\begin{bmatrix}
\frac{1}{t} & 0\\
0 & \frac{1}{t}%
\end{bmatrix}
.
\end{align*}
By (\ref{4.3}), we have%
\begin{align*}
e_{R}\left(  t,1\right)   &  =M^{\frac{1}{q}\Theta\left(  t\right)  }\\
&  =M^{\frac{1}{q}\left[  \delta_{-}\left(  1,q\right)  +\ldots+\delta
_{-}\left(  t_{m\left(  t\right)  -1},t_{m\left(  t\right)  }\right)  \right]
}\\
&  =M^{\frac{1}{q}qm\left(  t\right)  }=M^{m(t)}.
\end{align*}
Then, the matrix function $L$ which is $q$-periodic in shifts is obtained as
follows:%
\begin{align*}
L\left(  t\right)   &  =\Phi_{A}\left(  t,1\right)  e_{R}^{-1}\left(
t,1\right) \\
&  =\left[
\begin{array}
[c]{cc}%
t & 0\\
0 & t
\end{array}
\right]  \left[
\begin{array}
[c]{cc}%
q^{-m\left(  t\right)  } & 0\\
0 & q^{-m\left(  t\right)  }%
\end{array}
\right] \\
&  =\left[
\begin{array}
[c]{cc}%
t & 0\\
0 & t
\end{array}
\right]
\begin{bmatrix}
\frac{1}{t} & 0\\
0 & \frac{1}{t}%
\end{bmatrix}
=I
\end{align*}
since $q^{-m(t)}=q^{-n}=t^{-1}$ for $\mathbb{T}=\overline{q^{\mathbb{Z}}}$.
\end{example}

\begin{example}
\label{examp2}Suppose that $\mathbb{T}=\cup_{k=0}^{\infty}\left[  3^{\pm
k},2.3^{\pm k}\right]  \cup\left\{  0\right\}  $. Then, $\mathbb{T}$ is
$3$-periodic in shifts $\delta_{\pm}\left(  s,t\right)  =s^{\pm1}t$. If we set
$A\left(  t\right)  =1/t$, then we get
\[
A\left(  \delta_{\pm}\left(  3,t\right)  \right)  \delta_{\pm}^{\Delta}\left(
3,t\right)  =A\left(  3t\right)  3=\frac{1}{t}=A\left(  t\right)
\]
which shows that $A$ is $\Delta$-periodic in shifts with the period $3$.
Consider the system
\[
x^{\Delta}\left(  t\right)  =\left[
\begin{array}
[c]{cc}%
\frac{1}{t} & 0\\
0 & \frac{1}{t}%
\end{array}
\right]  x\left(  t\right)
\]
whose transition matrix is given by%
\[
\Phi_{A}\left(  t,1\right)  =\left[
\begin{array}
[c]{cc}%
e_{1/t}(t,1) & 0\\
0 & e_{1/t}(t,1)
\end{array}
\right]  .
\]
Then%
\[
\Phi_{A}\left(  \delta_{+}^{3}\left(  1\right)  ,1\right)  =\Phi_{A}\left(
3,1\right)  =\left[
\begin{array}
[c]{cc}%
e_{1/3}(3,1) & 0\\
0 & e_{1/3}(3,1)
\end{array}
\right]  .
\]
As in Theorem \ref{thm4.1}, we can write that%
\[
e_{R}\left(  3,1\right)  =\Phi_{A}\left(  3,1\right)  =\left[
\begin{array}
[c]{cc}%
e_{1/3}(3,1) & 0\\
0 & e_{1/3}(3,1)
\end{array}
\right]  =M.
\]
On the other hand, by (\ref{4.1.2}) and (\ref{4.3}) we have%
\begin{align*}
e_{R}\left(  t,1\right)   &  =M^{\frac{1}{3}\Theta\left(  t\right)  }\\
&  =\left\{
\begin{array}
[c]{cc}%
M^{\frac{1}{3}\left[  3m\left(  t\right)  -3^{m\left(  t\right)  }/t\right]  }
& \text{if }t\notin P\left(  1\right) \\
M^{\frac{1}{3}m\left(  t\right)  } & \text{if }t\in P\left(  1\right)
\end{array}
\right.  ,
\end{align*}
and%
\begin{align*}
R\left(  t\right)   &  =\lim_{s\rightarrow t}\frac{M^{\frac{1}{3}\left[
\Theta\left(  \sigma\left(  t\right)  \right)  -\Theta\left(  s\right)
\right]  }-I}{\sigma\left(  t\right)  -s}\\
&  =\left\{
\begin{array}
[c]{cc}%
\frac{2}{t}\left(  M^{\frac{1}{3}\left[  \Theta\left(  \frac{3}{2}t\right)
-\Theta\left(  t\right)  \right]  }-I\right)  & \text{if }\sigma\left(
t\right)  >t\\
\frac{1}{3}Log\left[  M\right]  & \text{if }\sigma\left(  t\right)  =t
\end{array}
\right.  ,
\end{align*}
where $P\left(  t\right)  $ and $m\left(  t\right)  $ are defined by
(\ref{P(t)}) and (\ref{m(t)}), respectively. Then we obtain the matrix
function $L\left(  t\right)  $ which is $3$-periodic in shifts as follows:
\begin{align*}
L\left(  t\right)   &  =\Phi_{A}\left(  t,1\right)  e_{R}^{-1}\left(
t,1\right) \\
&  =\left[
\begin{array}
[c]{cc}%
e_{1/t}(t,1) & 0\\
0 & e_{1/t}(t,1)
\end{array}
\right]  \left[
\begin{array}
[c]{cc}%
e_{1/3}(3,1) & 0\\
0 & e_{1/3}(3,1)
\end{array}
\right]  ^{-\frac{1}{3}\Theta\left(  t\right)  }.
\end{align*}

\end{example}

\begin{example}
Consider the time scale $\mathbb{T}=\mathbb{R}$ that is periodic in shifts
$\delta_{\pm}\left(  s,t\right)  =s^{\pm1}t$ associated with the initial point
$t_{0}=1$. Let us define the matrix function $A\left(  t\right)
:\mathbb{T}^{\ast}\mathbb{\rightarrow R}^{n\times n}$ as follows%
\[
A\left(  t\right)  =\left[
\begin{array}
[c]{cc}%
\frac{1}{t}\sin\left(  \pi\frac{\ln t}{\ln 2}\right)  & 0\\
0 & \frac{1}{t}\sin\left(  \pi\frac{\ln t}{\ln 2}\right)
\end{array}
\right]  .
\]
Then $A(t)$ is $\Delta$-periodic in shifts with the period $4.$ The following
system%
\[
x^{\Delta}\left(  t\right)  =\left[
\begin{array}
[c]{cc}%
\frac{1}{t}\sin\left(  \pi\frac{\ln t}{\ln 2}\right)  & 0\\
0 & \frac{1}{t}\sin\left(  \pi\frac{\ln t}{\ln 2}\right)
\end{array}
\right]  x\left(  t\right)
\]
has the transition matrix%
\[
\Phi_{A}\left(  t,1\right)  =\left[
\begin{array}
[c]{cc}%
e_{u(t)}(t,1) & 0\\
0 & e_{u(t)}(t,1)
\end{array}
\right]  ,
\]
where $u(t)=\frac{1}{t}\sin\left(  \pi\frac{\ln t}{\ln 2}\right)  $. Moreover,%
\[
\Phi_{A}\left(  \delta_{+}^{4}\left(  1\right)  ,1\right)  =\Phi_{A}\left(
4,1\right)  =\left[
\begin{array}
[c]{cc}%
1 & 0\\
0 & 1
\end{array}
\right]  =M.
\]
Thus, $R(t)$ is $2\times 2$ zero matrix, and hence, $e_{R}(t,1)=I$. Finally, the
matrix function $L(t)$ which is $4$-periodic in shifts is obtained as follows:%
\begin{align*}
L(t)  &  =\Phi_{A}\left(  t,1\right)  e_{R}^{-1}\left(  t,1\right) \\
&  =\Phi_{A}\left(  t,1\right)  .
\end{align*}

\end{example}

\subsection{Floquet multipliers and Floquet exponents}

In this section we investigate Floquet multipliers and exponents for the
system (\ref{4.1}). Let $\Phi_{A}\left(  t,t_{0}\right)  $ be the transition
matrix and $\Phi\left(  t\right)  $ the fundamental matrix at $t=\tau$ (i.e.
$\Phi\left(  \tau\right)  =I$) for the system (\ref{4.1}). Then, we can write
any fundamental matrix $\Psi\left(  t\right)  $ as follows
\begin{equation}
\Psi\left(  t\right)  =\Phi\left(  t\right)  \Psi\left(  \tau\right)  \text{
or }\Psi\left(  t\right)  =\Phi_{A}\left(  t,t_{0}\right)  \Psi\left(
t_{0}\right)  . \label{10}%
\end{equation}

\begin{definition}
\label{def4.2}Let $x_{0}\in\mathbb{R}^{n}$ be a nonzero vector and
$\Psi\left(  t\right)  $ be any fundamental matrix for the linear dynamic
system (\ref{4.1}). The vector solution of the system with initial condition
$x\left(  t_{0}\right)  =x_{0}$ is given by $\Phi_{A}\left(  t,t_{0}\right)
x_{0}$. We define the monodromy operator $M:\mathbb{R}^{n}\rightarrow
\mathbb{R}^{n}$ as follows:%
\begin{equation}
M\left(  x_{0}\right)  :=\Phi_{A}\left(  \delta_{+}^{T}\left(  t_{0}\right)
,t_{0}\right)  x_{0}=\Psi\left(  \delta_{+}^{T}\left(  t_{0}\right)  \right)
\Psi^{-1}\left(  t_{0}\right)  x_{0.} \label{11}%
\end{equation}
The eigenvalues of the monodromy operator are called Floquet multipliers of
the linear system (\ref{4.1}).
\end{definition}

Similar to \cite[Theorem 5.2 (i)]{[3]} we can give the following result:

\begin{remark}
\label{rem multipliers}The monodromy operator of the linear system (\ref{4.1})
is invertible. In particular, every characteristic multiplier is nonzero.
\end{remark}

\begin{theorem}
\label{thm6}The monodromy operator $M$ corresponding to different fundamental
matrices of the system (\ref{4.1}) is unique.
\end{theorem}

\begin{proof}
Suppose that $M_{1}$ and $M_{2}$ are the monodromy operators corresponding to
fundamental matrices $\Psi_{1}\left(  t\right)  $ and $\Psi_{2}\left(
t\right)  ,$ respectively. By using Definition \ref{def4.2}, we can express
the monodromy operator $M_{2}\left(  x_{0}\right)  $ corresponding to
$\Psi_{2}\left(  t\right)  $ as
\[
M_{2}\left(  x_{0}\right)  =\Psi_{2}\left(  \delta_{+}^{T}\left(
t_{0}\right)  \right)  \Psi_{2}^{-1}\left(  t_{0}\right)  x_{0}.
\]
Using (\ref{10}), we get
\begin{align*}
M_{2}\left(  x_{0}\right)   &  =\Psi_{2}\left(  \delta_{+}^{T}\left(
t_{0}\right)  \right)  \Psi_{2}^{-1}\left(  t_{0}\right)  x_{0}\\
&  =\Psi_{1}\left(  \delta_{+}^{T}\left(  t_{0}\right)  \right)  \Psi
_{2}\left(  \tau\right)  \Psi_{2}^{-1}\left(  \tau\right)  \Psi_{1}%
^{-1}\left(  t_{0}\right)  x_{0}\\
&  =\Psi_{1}\left(  \delta_{+}^{T}\left(  t_{0}\right)  \right)  \Psi_{1}%
^{-1}\left(  t_{0}\right)  x_{0}\\
&  =M_{1}\left(  x_{0}\right)  .
\end{align*}
The proof is complete.
\end{proof}

By using Theorem \ref{thm1}, (\ref{10}) and (\ref{11}), we obtain%
\begin{equation}
\Phi_{A}\left(  t,t_{0}\right)  =\Psi_{1}\left(  t\right)  \Psi_{1}%
^{-1}\left(  t_{0}\right)  =L\left(  t\right)  e_{R}\left(  t,t_{0}\right)
L^{-1}\left(  t_{0}\right)  \label{12}%
\end{equation}
and%
\begin{equation}
M\left(  x_{0}\right)  =\Phi_{A}\left(  \delta_{+}^{T}\left(  t_{0}\right)
,t_{0}\right)  x_{0}=\Psi_{1}\left(  \delta_{+}^{T}\left(  t_{0}\right)
\right)  \Psi_{1}^{-1}\left(  t_{0}\right)  x_{0}. \label{13}%
\end{equation}
If we combine (\ref{12}) and (\ref{13}), we get%
\[
\Phi_{A}\left(  \delta_{+}^{T}\left(  t_{0}\right)  ,t_{0}\right)  =\Psi
_{1}\left(  \delta_{+}^{T}\left(  t_{0}\right)  \right)  \Psi_{1}^{-1}\left(
t_{0}\right)  =L\left(  \delta_{+}^{T}\left(  t_{0}\right)  \right)
e_{R}\left(  \delta_{+}^{T}\left(  t_{0}\right)  ,t_{0}\right)  L^{-1}\left(
\delta_{+}^{T}\left(  t_{0}\right)  \right)  .
\]
By using the periodicity in shifts of $L,$ we have%
\begin{equation}
\Phi_{A}\left(  \delta_{+}^{T}\left(  t_{0}\right)  ,t_{0}\right)  =L\left(
t_{0}\right)  e_{R}\left(  \delta_{+}^{T}\left(  t_{0}\right)  ,t_{0}\right)
L^{-1}\left(  t_{0}\right)  . \label{14}%
\end{equation}
Hence, we arrive at the next result:

\begin{corollary}
\label{rem4.4}The Floquet multipliers of the system (\ref{4.1}) are the
eigenvalues of the matrix $e_{R}\left(  \delta_{+}^{T}\left(  t_{0}\right)
,t_{0}\right)  $.
\end{corollary}

\begin{definition}
[Floquet exponent]\label{def4.3}The Floquet exponent of the system (\ref{4.1})
is the function $\gamma\left(  t\right)  $ satisfying the equation%
\[
e_{\gamma}\left(  \delta_{+}^{T}\left(  t_{0}\right)  ,t_{0}\right)
=\lambda,
\]
where $\lambda$ is the Floquet multiplier of the system.
\end{definition}

\begin{definition}
\label{def4.4}\cite[Definition 2.4]{[1]} Let $\frac{-\pi}{h}<\omega\leq
\frac{\pi}{h}.$ Then Hilger purely imaginary number $\overset{\circ}{\imath}\omega$
is defined by $\overset{\circ}{\imath}\omega=\frac{e^{iwh}-1}{h}.$ For
$z\in\mathbb{C}_{h},$ we have that $\overset{\circ}{\imath}\operatorname{Im}%
_{h}\left(  z\right)  \in\mathbb{I}_{h}.$ Also, when $h=0,$ $\overset{\circ}{\imath}\omega=i\omega.$
\end{definition}

\begin{theorem}
\label{thm7}Suppose that $\gamma\left(  t\right)  \in\mathcal{R}$ is a Floquet
exponent of the system (\ref{4.1}) satisfying $e_{\gamma}\left(  \delta
_{+}^{T}\left(  t_{0}\right)  ,t_{0}\right)  =\lambda$, where $\lambda$ is
corresponding Floquet multiplier of the $T$-periodic system. Then
$\gamma\left(  t\right)  \oplus\overset{\circ}{\imath}\frac{2\pi k}{\delta
_{+}^{T}\left(  t_{0}\right)  -t_{0}}$ is also a Floquet exponent for
(\ref{4.1}) for all $k\in\mathbb{Z}$.
\end{theorem}

\begin{proof}
For all $k\in\mathbb{Z}$ and any $t_{0}\in\mathbb{T}^{\ast}$ we have%
\begin{align*}
e_{\gamma\oplus\overset{\circ}{\imath}\frac{2\pi k}{\delta_{+}^{T}\left(
t_{0}\right)  -t_{0}}}\left(  \delta_{+}^{T}\left(  t_{0}\right)
,t_{0}\right)   &  =e_{\gamma}\left(  \delta_{+}^{T}\left(  t_{0}\right)
,t_{0}\right)  e_{\overset{\circ}{\imath}\frac{2\pi k}{\delta_{+}^{T}\left(
t_{0}\right)  -t_{0}}}\left(  \delta_{+}^{T}\left(  t_{0}\right)
,t_{0}\right) \\
&  =e_{\gamma}\left(  \delta_{+}^{T}\left(  t_{0}\right)  ,t_{0}\right)
\exp\left(
{\displaystyle\int\limits_{t_{0}}^{\delta_{+}^{T}\left(  t_{0}\right)  }}
\frac{\log\left(  1+\mu\left(  \tau\right)  \overset{\circ}{\imath}\frac{2\pi
k}{\delta_{+}^{T}\left(  t_{0}\right)  -t_{0}}\right)  }{\mu\left(
\tau\right)  }\Delta\tau\right) \\
&  =e_{\gamma}\left(  \delta_{+}^{T}\left(  t_{0}\right)  ,t_{0}\right)
\exp\left(
{\displaystyle\int\limits_{t_{0}}^{\delta_{+}^{T}\left(  t_{0}\right)  }}
\frac{\log\left(  \exp\left(  i\frac{2\pi k\mu\left(  \tau\right)  }%
{\delta_{+}^{T}\left(  t_{0}\right)  -t_{0}}\right)  \right)  }{\mu\left(
\tau\right)  }\Delta\tau\right) \\
&  =e_{\gamma}\left(  \delta_{+}^{T}\left(  t_{0}\right)  ,t_{0}\right)
\exp\left(
{\displaystyle\int\limits_{t_{0}}^{\delta_{+}^{T}\left(  t_{0}\right)  }}
\frac{i2\pi k}{\delta_{+}^{T}\left(  t_{0}\right)  -t_{0}}\Delta\tau\right) \\
&  =e_{\gamma}\left(  \delta_{+}^{T}\left(  t_{0}\right)  ,t_{0}\right)
e^{i2\pi k}\\
&  =e_{\gamma}\left(  \delta_{+}^{T}\left(  t_{0}\right)  ,t_{0}\right)  ,
\end{align*}
which gives the desired result.
\end{proof}

The next result can be proven similar to \cite[Theorem 5.3]{[3]}.

\begin{theorem}
\label{thm8} Let $R\left(  t\right)  $ be a matrix function as in Theorem
\ref{thm4.1}, with eigenvalues $\gamma_{1}\left(  t\right)  ,\ldots,\gamma
_{n}\left(  t\right)  $ repeated according to multiplicities. Then $\gamma
_{1}^{k}\left(  t\right)  ,\ldots,\gamma_{n}^{k}\left(  t\right)  $ are the
eigenvalues of $R^{k}\left(  t\right)  $ and eigenvalues of $e_{R}$ are
$e_{\gamma_{1}},\ldots,e_{\gamma_{n}}$.
\end{theorem}

\begin{lemma}
Let $\mathbb{T}$ be a time scale that is $p$-periodic in shifts $\delta_{\pm}$
associated with the initial point $t_{0}$ and $k\in\mathbb{Z}$. If
$\frac{\delta_{+}^{p}\left(  t\right)  -t}{\delta_{+}^{p}\left(  t_{0}\right)
-t_{0}}\in\mathbb{Z}$, then the functions $e_{\overset{\circ}{\imath}%
\frac{2\pi k}{\delta_{+}^{T}\left(  t_{0}\right)  -t_{0}}}$ and $e_{\ominus
\overset{\circ}{\imath}\frac{2\pi k}{\delta_{+}^{T}\left(  t_{0}\right)
-t_{0}}}$ are $p$ periodic in shifts.
\end{lemma}

\begin{proof}
If $\frac{\delta_{+}^{p}\left(  t\right)  -t}{\delta_{+}^{p}\left(
t_{0}\right)  -t_{0}}\in\mathbb{Z}$, then we have%
\begin{align*}
e_{\overset{\circ}{\imath}\frac{2\pi k}{\delta_{+}^{T}\left(  t_{0}\right)
-t_{0}}}\left(  \delta_{+}^{p}\left(  t\right)  ,t_{0}\right)   &
=\exp\left(
{\displaystyle\int\limits_{t_{0}}^{\delta_{+}^{p}\left(  t\right)  }}
\frac{i2\pi k}{\delta_{+}^{p}\left(  t_{0}\right)  -t_{0}}\Delta\tau\right) \\
&  =\exp\left(
{\displaystyle\int\limits_{t}^{\delta_{+}^{p}\left(  t\right)  }}
\frac{i2\pi k}{\delta_{+}^{p}\left(  t_{0}\right)  -t_{0}}\Delta\tau\right)
\exp\left(
{\displaystyle\int\limits_{t_{0}}^{t}}
\frac{i2\pi k}{\delta_{+}^{p}\left(  t_{0}\right)  -t_{0}}\Delta\tau\right) \\
&  =\exp\left(  i2\pi k\frac{\delta_{+}^{p}\left(  t\right)  -t}{\delta
_{+}^{p}\left(  t_{0}\right)  -t_{0}}\right)  \exp\left(
{\displaystyle\int\limits_{t_{0}}^{t}}
\frac{i2\pi k}{\delta_{+}^{p}\left(  t_{0}\right)  -t_{0}}\Delta\tau\right) \\
&  =\exp\left(
{\displaystyle\int\limits_{t_{0}}^{t}}
\frac{i2\pi k}{\delta_{+}^{p}\left(  t_{0}\right)  -t_{0}}\Delta\tau\right)
=e_{\overset{\circ}{\imath}\frac{2\pi k}{\delta_{+}^{T}\left(  t_{0}\right)
-t_{0}}}\left(  t,t_{0}\right)
\end{align*}
which proves the periodicity of $e_{\overset{\circ}{\imath}\frac{2\pi
k}{\delta_{+}^{T}\left(  t_{0}\right)  -t_{0}}}$. The periodicity of
$e_{\ominus\overset{\circ}{\imath}\frac{2\pi k}{\delta_{+}^{T}\left(
t_{0}\right)  -t_{0}}}$ can be proven by using the periodicity of
$e_{\overset{\circ}{\imath}\frac{2\pi k}{\delta_{+}^{T}\left(  t_{0}\right)
-t_{0}}}$ and the identity $e_{\ominus\alpha}=1/e_{\alpha}$.
\end{proof}

\begin{remark}
Note that the condition $\frac{\delta_{+}^{p}\left(  t\right)  -t}%
{\delta_{+}^{p}\left(  t_{0}\right)  -t_{0}}\in\mathbb{Z}$ holds not only for
all additive periodic time scales but also for the many time scales that are
periodic in shifts. For example for the two-periodic time scales
$\overline{2^{\mathbb{Z}}}$ and $\cup_{k=0}^{\infty}\left[  2^{\pm k}%
,2^{\pm(k+1)}\right]  \cup\left\{  0\right\}  $ in shifts $\delta_{\pm}\left(
s,t\right)  =s^{\pm1}t$ associated with the initial point $t_{0}=1$, the
condition $\frac{\delta_{+}^{p}\left(  t\right)  -t}{\delta_{+}^{p}\left(
t_{0}\right)  -t_{0}}\in\mathbb{Z}$ is always satisfied.
\end{remark}

\begin{theorem}
\label{thm9}If $\gamma\left(  t\right)  $ is a Floquet exponent for the system
(\ref{4.1}) and $\Phi_{A}\left(  t,t_{0}\right)  $ is the associated
transition matrix, then there exists a Floquet decomposition of the form%
\[
\Phi_{A}\left(  t,t_{0}\right)  =L\left(  t\right)  e_{R}\left(
t,t_{0}\right)
\]
such that $\gamma\left(  t\right)  $ is an eigenvalue of $R\left(  t\right)
.$
\end{theorem}

\begin{proof}
Consider the Floquet decomposition $\Phi_{A}\left(  t,t_{0}\right)
=\widetilde{L}\left(  t\right)  e_{\widetilde{R}}\left(  t,t_{0}\right)  $. By
Definition \ref{def4.3}, there exists a characteristic multiplier $\lambda$
such that $e_{\gamma}\left(  \delta_{+}^{T}(t_{0}),t_{0}\right)  =\lambda$.
Moreover, there is an eigenvalue $\tilde{\gamma}\left(  t\right)  $ of
$\tilde{R}\left(  t\right)  $ so that $e_{\tilde{\gamma}}\left(  \delta
_{+}^{T}(t_{0}),t_{0}\right)  =\lambda$, where $\tilde{\gamma}\left(
t\right)  $ can be defined as
\[
\tilde{\gamma}\left(  t\right)  :=\gamma\left(  t\right)  \oplus\overset
{\circ}{\imath}\frac{2\pi k}{\delta_{+}^{T}\left(  t_{0}\right)  -t_{0}}%
\]
by Theorem \ref{thm7}. If we set%
\[
R\left(  t\right)  :=\widetilde{R}\left(  t\right)  \ominus\overset{\circ
}{\imath}\frac{2\pi k}{\delta_{+}^{T}\left(  t_{0}\right)  -t_{0}}I%
\]
and%
\[
L\left(  t\right)  :=\tilde{L}\left(  t\right)  e_{\overset{\circ}{\imath
}\frac{2\pi k}{\delta_{+}^{T}\left(  t_{0}\right)  -t_{0}}I}\left(
t,t_{0}\right)  ,
\]
then we can write%
\[
\widetilde{R}\left(  t\right)  :=R\left(  t\right)  \oplus\overset{\circ
}{\imath}\frac{2\pi k}{\delta_{+}^{T}\left(  t_{0}\right)  -t_{0}}I,
\]
and hence,%
\[
L\left(  t\right)  e_{R}\left(  t,t_{0}\right)  =\tilde{L}\left(  t\right)
e_{\overset{\circ}{\imath}\frac{2\pi k}{\delta_{+}^{T}\left(  t_{0}\right)
-t_{0}}I}\left(  t,t_{0}\right)  e_{R}\left(  t,t_{0}\right)  =\tilde
{L}\left(  t\right)  e_{\overset{\circ}{\imath}\frac{2\pi k}{\delta_{+}%
^{T}\left(  t_{0}\right)  -t_{0}}I\oplus R}\left(  t,t_{0}\right)  =\tilde
{L}\left(  t\right)  e_{\tilde{R}}\left(  t,t_{0}\right)  .
\]
This means $\Phi_{A}\left(  t,t_{0}\right)  =L\left(  t\right)  e_{R}\left(
t,t_{0}\right)  $ is another Floquet decomposition where $\gamma\left(
t\right)  $ is an eigenvalue of $R\left(  t\right)  .$
\end{proof}

\begin{theorem}
\label{thm10}Suppose that $\lambda$ is a characteristic multiplier of the
system (\ref{4.1}) and that $\gamma\left(  t\right)  $ is the corresponding
Floquet exponent. Then, (\ref{4.1}) has a nontrivial solution of the form%
\begin{equation}
x\left(  t\right)  =e_{\gamma}\left(  t,t_{0}\right)  q\left(  t\right)
\label{xq}%
\end{equation}
satisfying%
\[
x\left(  \delta_{+}^{T}\left(  t\right)  \right)  =\lambda x\left(  t\right)
,
\]
where $q$ is a $T$-periodic function in shifts.
\end{theorem}

\begin{proof}
Let $\Phi_{A}\left(  t,t_{0}\right)  $ be the transition matrix of (\ref{4.1})
and $\Phi_{A}\left(  t,t_{0}\right)  =L\left(  t\right)  e_{R}\left(
t,t_{0}\right)  $ is Floquet decomposition such that $\gamma\left(  t\right)
$ is an eigenvalue of $R\left(  t\right)  $. There exists a nonzero vector
$u\neq0$ such that $R\left(  t\right)  u=\gamma\left(  t\right)  u$, and
therefore, $e_{R}\left(  t,t_{0}\right)  u=e_{\gamma}\left(  t,t_{0}\right)
u$. Then, we can represent the solution $x\left(  t\right)  :=\Phi_{A}\left(
t,t_{0}\right)  u$ as follows
\[
x\left(  t\right)  =L\left(  t\right)  e_{R}\left(  t,t_{0}\right)
u=e_{\gamma}\left(  t,t_{0}\right)  L\left(  t\right)  u.
\]
If we set $q\left(  t\right)  =L\left(  t\right)  u$, the last equality
implies (\ref{xq}). Thus, the first part of the theorem is proven.

The second part is proven by the following equality.%
\begin{align*}
x\left(  \delta_{+}^{T}\left(  t\right)  \right)   &  =e_{\gamma}\left(
\delta_{+}^{T}\left(  t\right)  ,t_{0}\right)  q\left(  \delta_{+}^{T}\left(
t\right)  \right) \\
&  =e_{\gamma}\left(  \delta_{+}^{T}\left(  t\right)  ,\delta_{+}^{T}\left(
t_{0}\right)  \right)  e_{\gamma}\left(  \delta_{+}^{T}\left(  t_{0}\right)
,t_{0}\right)  q\left(  t\right) \\
&  =e_{\gamma}\left(  \delta_{+}^{T}\left(  t_{0}\right)  ,t_{0}\right)
e_{\gamma}\left(  t,t_{0}\right)  L\left(  t\right)  u\\
&  =e_{\gamma}\left(  \delta_{+}^{T}\left(  t_{0}\right)  ,t_{0}\right)
x\left(  t\right) \\
&  =\lambda x\left(  t\right)  .
\end{align*}

\end{proof}

The preceding theorem provides a procedure for the construction of a solution
to the system (\ref{4.1}) when a characteristic multiplier is given. In the
following theorem, we show that two solutions corresponding to two distinct
characteristic multipliers are linearly independent.

\begin{theorem}
\label{co1}Let $\lambda_{1}$ and $\lambda_{2}$ be the characteristic
multipliers of the system (\ref{4.1}) and $\gamma_{1}$ and $\gamma_{2}$ are
Floquet exponents such that%
\[
e_{\gamma_{i}}(\delta_{+}^{T}(t_{0}),t_{0})=\lambda_{i}\text{,\ \ \ }i=1,2.
\]
If $\lambda_{1}\neq\lambda_{2}$, then there exist $T$-periodic functions
$q_{1}$ and $q_{2}$ in shifts such that%
\[
x_{i}(t)=e_{\gamma_{i}}(t,t_{0})q_{i}(t),\text{ }i=1,2
\]
are linearly independent solutions of (\ref{4.1}).
\end{theorem}

\begin{proof}
Let $\Phi_{A}\left(  t,t_{0}\right)  =L\left(  t\right)  e_{R}\left(
t,t_{0}\right)  $ and $\gamma_{1}\left(  t\right)  $ be an eigenvalue of
$R\left(  t\right)  $ corresponding to nonzero eigenvector $v_{1}.$ Since
$\lambda_{2}$ is an eigenvalue of $\Phi_{A}\left(  \delta_{+}^{T}\left(
t_{0}\right)  ,t_{0}\right)  $, by Theorem \ref{thm8} there is an eigenvalue
$\gamma\left(  t\right)  $\ of $\ R\left(  t\right)  $ satisfying%
\[
e_{\gamma}\left(  \delta_{+}^{T}\left(  t_{0}\right)  ,t_{0}\right)
=\lambda_{2}=e_{\gamma_{2}}\left(  \delta_{+}^{T}\left(  t_{0}\right)
,t_{0}\right)  \text{.}%
\]
Hence, for some $k\in\mathbb{Z}$ we have $\gamma_{2}(t)=\gamma\left(
t\right)  \oplus\overset{\circ}{\imath}\frac{2\pi k}{\delta_{+}^{T}\left(
t_{0}\right)  -t_{0}}$. Furthermore, $\lambda_{1}\neq\lambda_{2}$ implies that
$\gamma(t)\neq\gamma_{1}\left(  t\right)  $. If $v_{2}$ is a nonzero
eigenvector of $R\left(  t\right)  $ corresponding to eigenvalue $\gamma(t)$,
then the eigenvectors $v_{1}$ and $v_{2}$ are linearly independent. Similar to
the related part in the proof of Theorem \ref{thm10}, we can state the
solutions of the system (\ref{4.1}) as follows:%
\begin{equation}
x_{1}\left(  t\right)  =e_{\gamma_{1}}\left(  t,t_{0}\right)  L\left(
t\right)  v_{1} \label{4.4}%
\end{equation}
and
\[
x_{2}\left(  t\right)  =e_{\gamma}\left(  t,t_{0}\right)  L\left(  t\right)
v_{2}.
\]
Since $x_{1}\left(  t_{0}\right)  =L(t_{0})v_{1}$ and $x_{2}\left(  t_{0}\right)
=L(t_{0})v_{2}$, the solutions $x_{1}\left(  t\right)  $ and $x_{2}\left(  t\right)  $
are linearly independent. Moreover, the solution $x_{2}$ can be rewritten in
the following form%
\begin{align}
x_{2}(t)  &  =e_{\gamma_{2}}\left(  t,t_{0}\right)  e_{\gamma\ominus\gamma
_{2}}(t,t_{0})L(t)\nu_{2}\nonumber\\
&  =e_{\gamma_{2}}\left(  t,t_{0}\right)  e_{\ominus\overset{\circ}{\imath
}\frac{2\pi k}{\delta_{+}^{T}\left(  t_{0}\right)  -t_{0}}}\left(
t,t_{0}\right)  L(t)\nu_{2}. \label{4.5}%
\end{align}
Letting $q_{1}\left(  t\right)  =L\left(  t\right)  v_{1}$ and $q_{2}\left(
t\right)  =e_{\ominus\overset{\circ}{\imath}\frac{2\pi k}{\delta_{+}%
^{T}\left(  t_{0}\right)  -t_{0}}}\left(  t,t_{0}\right)  L(t)\nu_{2}$ in
(\ref{4.4}) and (\ref{4.5}), respectively, we complete the proof.
\end{proof}

\section{Floquet Theory and Stability}

In this section, we employ the unified Floquet theory that we established in
previous sections to investigate the stability characteristics of the
regressive periodic system
\begin{equation}
x^{\Delta}\left(  t\right)  =A\left(  t\right)  x\left(  t\right)  ,\text{
}x\left(  t_{0}\right)  =x_{0}. \label{5.0}%
\end{equation}
We know by Theorem \ref{thm4.1} that the matrix $R$ in the Floquet
decomposition of $\Phi_{A}$ is given by%
\begin{equation}
R\left(  t\right)  =\lim_{s\rightarrow t}\frac{\Phi_{A}\left(  \delta_{+}%
^{T}\left(  t_{0}\right)  ,t_{0}\right)  ^{\frac{1}{T}\left[  \Theta\left(
\sigma\left(  t\right)  \right)  -\Theta\left(  s\right)  \right]  }-I}%
{\sigma\left(  t\right)  -s}. \label{5.0.1}%
\end{equation}
Also, Theorem \ref{thm2} concludes that the solution $z(t)$ of the
regressive system
\begin{equation}
z^{\Delta}\left(  t\right)  =R\left(  t\right)  z\left(  t\right)  ,\text{
}z\left(  t_{0}\right)  =x_{0} \label{5.1}%
\end{equation}
can be expressed in terms of the solution $x(t)$ of the system (\ref{5.0}) as
follows: $z(t)=L^{-1}(t)x(t)$, where $L(t)$ is the Lyapunov transformation
given by (\ref{3.6}).

In preparation for the main result we can give the following definitions and
results which can be found in \cite{[3]}.

\begin{definition}
[Stability]The time varying linear dynamic equation (\ref{5.0}) is uniformly
stable if there exists a positive constant $\alpha$ such that for any $t_{0}$
the corresponding solution $x(t)$ satisfies%
\[
\left\Vert x(t)\right\Vert \leq\alpha\left\Vert x(t_{0})\right\Vert ,\text{
}t\geq t_{0}.
\]

\end{definition}

\begin{theorem}
The time varying linear dynamic equation (\ref{5.0}) is uniformly stable if
and only if there exists a $\alpha>0$ such that the transition matrix
$\Phi_{A}$ satisfies%
\[
\left\Vert \Phi_{A}(t,t_{0})\right\Vert \leq\alpha,\text{ }t\geq t_{0}.
\]

\end{theorem}

\begin{definition}
[Exponential stability]The time varying linear dynamic equation (\ref{5.0}) is
uniformly exponentially stable if there exist positive constants $\alpha$,
$\beta$ with $-\beta\in\mathcal{R}^{+}$ such that for any $t_{0}$ the
corresponding solution $x(t)$ satisfies%
\[
\left\Vert x(t)\right\Vert \leq\left\Vert x(t_{0})\right\Vert \alpha
e_{-\beta}(t,t_{0}),\text{ }t\geq t_{0}.
\]

\end{definition}

Moreover, necessary and sufficient conditions for exponential stability can be
stated as the following:

\begin{theorem}
The time varying linear dynamic equation (\ref{5.0}) is uniformly
exponentially stable if and only if there exist $\alpha$, $\beta>0$ with
$-\beta\in\mathcal{R}^{+}$ such that the transition matrix $\Phi_{A}$
satisfies%
\[
\left\Vert \Phi_{A}(t,t_{0})\right\Vert \leq\alpha e_{-\beta}(t,t_{0}),\text{
}t\geq t_{0}.
\]

\end{theorem}

\begin{definition}
[Asymptotical stability]The system (\ref{5.0}) is said to be uniformly
asymptotically stable if it is uniformly stable and given any $c>0,$ there
exists a $K>0$ so that for any $t_{0}$ and $x(t_{0}),$ the corresponding
solution $x(t)$ satisfies%
\[
\left\Vert x(t)\right\Vert \leq c\left\Vert x(t_{0})\right\Vert ,\text{ }t\geq
t_{0}+K.
\]

\end{definition}

Given a constant $n\times n$ matrix $M$, let $S$ be a nonsingular matrix that
transforms $M$ into its Jordan canonical form%
\[
J:=S^{-1}MS=diag\left[  J_{m_{1}}\left(  \lambda_{1}\right)  ,\ldots,J_{m_{k}%
}\left(  \lambda_{k}\right)  \right]  ,
\]
where $k\leq n,$ $%
{\displaystyle\sum\limits_{i=1}^{k}}
m_{i}=n,$ $\lambda_{i}$ are the eigenvalues of $M$, and $J_{m}\left(
\lambda\right)  $ is an $m\times m$ Jordan block given by%
\[
J_{m}\left(  \lambda\right)  =\left[
\begin{array}
[c]{ccccc}%
\lambda & 1 &  &  & \\
& \lambda & 1 &  & \\
&  & \ddots & \ddots & \\
&  &  & \ddots & 1\\
&  &  &  & \lambda
\end{array}
\right]  .
\]

\begin{definition}
(\label{def5.1} \cite{[19]} See also \cite[Definition 7.1]{[3]}) The scalar
function $\gamma:\mathbb{T\rightarrow C}$ is uniformly regressive if there
exists a constant $\theta>0$ such that $0<\theta^{-1}\leq\left\vert
1+\mu\left(  t\right)  \gamma\left(  t\right)  \right\vert ,$ for all
$t\in\mathbb{T}^{\kappa}.$
\end{definition}

\begin{lemma}
\label{lem5.2}Each eigenvalue of the matrix $R\left(  t\right)  $ in
(\ref{5.1}) is uniformly regressive.
\end{lemma}

\begin{proof}
Define $\Lambda(t,s)$ by
\begin{equation}
\Lambda(t,s):=\Theta\left(  \sigma\left(  t\right)  \right)  -\Theta\left(
s\right)  . \label{gamma t-s}%
\end{equation}
As we did in Corollary \ref{cor1}, let%
\[
\gamma_{i}\left(  t\right)  =\lim_{s\rightarrow t}\left(  \frac{\lambda
_{i}^{\frac{1}{T}\Lambda(t,s)}-1}{\sigma\left(  t\right)  -s}\right)  \text{,
}i=1,2,...,k
\]
be any of the $k\leq n$ distinct eigenvalues of $R\left(  t\right)  $. Now,
there are two cases:

\begin{enumerate}
\item If $\left\vert \lambda_{i}\right\vert \geq1$, then%
\[
\left\vert 1+\mu\left(  t\right)  \gamma_{i}\left(  t\right)  \right\vert
=\lim_{s\rightarrow t}\left\vert 1+\mu\left(  s\right)  \frac{\lambda
_{i}^{\frac{1}{T}\Lambda(t,s)}-1}{\sigma\left(  t\right)  -s}\right\vert
=\lim_{s\rightarrow t}\left\vert \lambda_{i}^{\frac{1}{T}\Lambda
(t,s)}\right\vert >1.
\]

\item If $0\leq\left\vert \lambda_{i}\right\vert <1$, then,%
\[
\left\vert 1+\mu\left(  t\right)  \gamma_{i}\left(  t\right)  \right\vert
=\lim_{s\rightarrow t}\left\vert 1+\mu\left(  s\right)  \frac{\lambda
_{i}^{\frac{1}{T}\Lambda(t,s)}-1}{\sigma\left(  t\right)  -s}\right\vert
=\lim_{s\rightarrow t}\left\vert \lambda_{i}^{\frac{1}{T}\Lambda
(t,s)}\right\vert \geq\left\vert \lambda_{i}\right\vert .
\]

\end{enumerate}

\noindent If we set $\theta^{-1}:=\min\{1,\left\vert \lambda_{1}\right\vert
,\ldots,\left\vert \lambda_{k}\right\vert \}$, then we obtain%
\[
0<\theta^{-1}<\left\vert 1+\mu\left(  t\right)  \gamma_{i}\left(  t\right)
\right\vert ,
\]
where we used Remark \ref{rem multipliers} to get $0<\theta^{-1}$.
\end{proof}

\begin{definition}
\label{def5.3}\cite[Definition 7.3]{[3]} A nonzero, delta differentiable
vector $w\left(  t\right)  $ is said to be a dynamic eigenvector of a matrix
$H\left(  t\right)  $ associated with the dynamic eigenvalue $\xi\left(
t\right)  $ if the pair satisfies the dynamic eigenvalue problem%
\begin{equation}
w^{\Delta}\left(  t\right)  =H\left(  t\right)  w\left(  t\right)  -\xi\left(
t\right)  w^{\sigma}\left(  t\right)  ,\text{ }t\in\mathbb{T}^{k}. \label{5.2}%
\end{equation}
We call $\{\xi\left(  t\right)  ,w\left(  t\right)  \}$ a dynamic eigenpair.
Also, the nonzero, delta differentiable vector%
\begin{equation}
\chi_{i}:=e_{\xi_{i}}\left(  t,t_{0}\right)  w_{i}\left(  t\right)  ,
\label{5.3}%
\end{equation}
is called the mode vector of $M\left(  t\right)  $ associated with the dynamic
eigenpair\ $\left\{  \xi_{i}\left(  t\right)  ,w_{i}\left(  t\right)
\right\}  $.
\end{definition}

Now, we can give the following results similar to \cite[Lemma 7.4, Theorem
7.5]{[3]}:

\begin{lemma}
\label{lem5.4}Given the $n\times n$ regressive matrix $K$, there always exists
a set of $n$ dynamic eigenpairs with linearly independent eigenvectors. Each
of the eigenpairs satisfies the vector dynamic eigenvalue problem (\ref{5.2})
associated with $H$. Furthermore, when the $n$ vectors form the columns of
$W\left(  t\right)  ,$ then $W\left(  t\right)  $ satisfies the equivalent
matrix dynamic eigenvalue problem%
\begin{equation}
W^{\Delta}\left(  t\right)  =H\left(  t\right)  W\left(  t\right)  -W^{\sigma
}\left(  t\right)  \Xi\left(  t\right)  ,\text{ where }\Xi\left(  t\right)
:=diag\left[  \xi_{1}\left(  t\right)  ,\ldots,\xi_{n}\left(  t\right)
\right]  . \label{5.4}%
\end{equation}

\end{lemma}

\begin{theorem}
\label{thm5.4.1}Solutions to the uniformly regressive (but not necessarily
periodic) time varying linear dynamic system (\ref{5.0}) are:

\begin{enumerate}
\item stable if and only if there exists a $\gamma>0$ such that every mode
vector $\chi_{i}\left(  t\right)  $ of $A\left(  t\right)  $ satisfies
$\left\Vert \chi_{i}\left(  t\right)  \right\Vert \leq\gamma<\infty$,
$t>t_{0}$, for all $1\leq i\leq n;$

\item asymptotically stable if and only if, in addition to (1), $\left\Vert
\chi_{i}\left(  t\right)  \right\Vert \rightarrow0$, $t>t_{0}$, for all $1\leq
i\leq n,$

\item exponentially stable if and only if there exists $\gamma,\lambda>0$ with
$-\lambda\in\mathcal{R}^{+}\left(  \mathbb{T},\mathbb{R}\right)  $ such that
$\left\Vert \chi_{i}\left(  t\right)  \right\Vert \leq\gamma e_{\lambda
}\left(  t,t_{0}\right)  $, $t>t_{0}$, for all $1\leq i\leq n.$
\end{enumerate}
\end{theorem}

\begin{definition}
\label{monomial} For each $k\in\mathbb{N}_{0}$ the mappings $h_{k}%
:\mathbb{T\times T}^{k}\rightarrow\mathbb{R}^{+}$, recursively defined by%
\begin{equation}
h_{0}(t,t_{0}):\equiv1,\text{ \ }h_{k+1}(t,t_{0})=%
{\displaystyle\int\limits_{t_{0}}^{t}}
\left(  \lim_{s\rightarrow\tau}\frac{\Lambda(\tau,s)}{\sigma(\tau)-s}\right)
h_{k}(\tau,t_{0})\Delta\tau\text{ for }n\in\mathbb{N}_{0},
\label{hk recursive}%
\end{equation}
are called monomials, where $\Lambda(t,s)$ is given by (\ref{gamma t-s}).
\end{definition}

\begin{remark}
For an additive periodic time scale we always have $\Theta\left(  t\right)
=t-t_{0}$, and hence, $\Lambda(t,s)=\sigma(t)-s$.
\end{remark}

\begin{lemma}
\label{lem5.6}Let $\mathbb{T}$ be a time scale which is unbounded above and
$\gamma\left(  t\right)  $ be an eigenvalue of $R(t).$ If there exists a
constant $H\geq t_{0}$ such that
\begin{equation}
\inf_{t\in\lbrack H,\infty)_{\mathbb{T}}}\left[  -\left(  \lim_{s\rightarrow
t}\left(  \frac{\Lambda(t,s)}{\sigma(t)-s}\right)  \right)  ^{-1}%
\operatorname{Re}_{\mu}\gamma\left(  t\right)  \right]  >0
\label{infimum condtion}%
\end{equation}
holds, then%
\begin{equation}
\lim_{t\rightarrow\infty}h_{k}\left(  t,t_{0}\right)  e_{\gamma}\left(
t,t_{0}\right)  =0,\text{ }k\in\mathbb{N}_{0}. \label{lim hk}%
\end{equation}

\begin{proof}
It suffices to show that $\lim_{t\rightarrow\infty}h_{k}\left(  t,t_{0}%
\right)  e_{\operatorname{Re}_{\mu}\gamma(t)}\left(  t,t_{0}\right)  =0$ (see
\cite[Theorem 7.4]{hilger}). We proceed by mathematical induction. For $k=0,$
we know that $h_{0}\left(  t,t_{0}\right)  =1$ and by \cite{[19]}, we have%
\[
\lim_{t\rightarrow\infty}e_{\operatorname{Re}_{\mu}\gamma_{i}(t)}\left(
t,t_{0}\right)  =0\text{ for }t_{0}\in\mathbb{T}.
\]
Suppose that it is true for a fixed $k\in\mathbb{N}$ and focus on the
$(k+1)^{th}$ step.
\begin{align}
&  \lim_{t\rightarrow\infty}h_{k+1}\left(  t,t_{0}\right)
e_{\operatorname{Re}_{\mu}\gamma(t)}\left(  t,t_{0}\right) \nonumber\\
&  =\lim_{t\rightarrow\infty}\left[
{\displaystyle\int\limits_{t_{0}}^{t}}
\mathfrak{R}\lim_{s\rightarrow\tau}\left(  \frac{\Lambda(\tau,s)}{\sigma
(\tau)-s}\right)  h_{k}\left(  \tau,t_{0}\right)  \Delta\tau+%
{\displaystyle\int\limits_{t_{0}}^{t}}
\mathfrak{\dot{I}}\lim_{s\rightarrow\tau}\left(  \frac{\Lambda(\tau,s)}%
{\sigma(\tau)-s}\right)  h_{k}\left(  \tau,t_{0}\right)  \Delta\tau\right]
e_{\ominus\operatorname{Re}_{\mu}\gamma(t)}\left(  t,t_{0}\right)
^{-1}\nonumber\\
&  =\lim_{t\rightarrow\infty}\left[  \mathfrak{R}\lim_{s\rightarrow t}\left(
\frac{\Lambda(t,s)}{\sigma(t)-s}\right)  h_{k}\left(  t,t_{0}\right)
+\mathfrak{\dot{I}}\lim_{s\rightarrow t}\left(  \frac{\Lambda(t,s)}%
{\sigma(t)-s}\right)  h_{k}\left(  t,t_{0}\right)  \right]  \frac
{e_{\operatorname{Re}_{\mu}\gamma(t)}\left(  t,t_{0}\right)  }{\ominus
\operatorname{Re}_{\mu}\gamma(t)}\nonumber\\
&  =\lim_{t\rightarrow\infty}\left[  \frac{\lim_{s\rightarrow t}\left(
\frac{\Lambda(t,s)}{\sigma(t)-s}\right)  h_{k}\left(  t,t_{0}\right)
e_{\operatorname{Re}_{\mu}\gamma(t)}\left(  t,t_{0}\right)  }{\ominus
\operatorname{Re}_{\mu}\gamma(t)}\right]  , \label{xyz}%
\end{align}
where we used (\ref{infimum condtion}) together with \cite[Theorem 1.120]{[1]}
to obtain the second equality. Since
\[
\ominus\operatorname{Re}_{\mu}\gamma_{i}(t)=\frac{-\operatorname{Re}_{\mu
}\gamma(t)}{1+\mu(t)\operatorname{Re}_{\mu}\gamma(t)},
\]
the last term in (\ref{xyz}) can be written as%
\begin{align}
&  \lim_{t\rightarrow\infty}\left[  \frac{\lim_{s\rightarrow t}\left(
\frac{\Lambda(t,s)}{\sigma(t)-s}\right)  h_{k}\left(  t,t_{0}\right)
e_{\operatorname{Re}_{\mu}\gamma(t)}\left(  t,t_{0}\right)  }{\ominus
\operatorname{Re}_{\mu}\gamma(t)}\right] \nonumber\\
&  =\lim_{t\rightarrow\infty}\left[  \frac{\left(  1+\mu(t)\operatorname{Re}%
_{\mu}\gamma(t)\right)  h_{k}\left(  t,t_{0}\right)  e_{\operatorname{Re}%
_{\mu}\gamma(t)}\left(  t,t_{0}\right)  }{-\left(  \lim_{s\rightarrow
t}\left(  \frac{\Lambda(t,s)}{\sigma(t)-s}\right)  \right)  ^{-1}%
\operatorname{Re}_{\mu}\left(  \gamma\left(  t\right)  \right)  }\right]
\nonumber\\
&  \leq\lim_{t\rightarrow\infty}\left[  \frac{\left(  1+\mu
(t)\operatorname{Re}_{\mu}\gamma(t)\right)  h_{k}\left(  t,t_{0}\right)
e_{\operatorname{Re}_{\mu}\gamma(t)}\left(  t,t_{0}\right)  }{\inf
_{t\in\lbrack H,\infty)_{\mathbb{T}}}\left[  -\left(  \lim_{s\rightarrow
t}\left(  \frac{\Lambda(t,s)}{\sigma(t)-s}\right)  \right)  ^{-1}%
\operatorname{Re}_{\mu}\left(  \gamma\left(  t\right)  \right)  \right]
}\right]  . \label{xyz0}%
\end{align}
Now, one may use (\ref{4.1.1}) and (\ref{gamma t-s}) to get the inequality
\[
1+\mu(t)\operatorname{Re}_{\mu}\gamma(t)=\left\vert 1+\mu(t)\lim_{s\rightarrow
t}\left(  \frac{\lambda^{\frac{1}{T}\Lambda(t,s)}-1}{\sigma\left(  t\right)
-s}\right)  \right\vert \leq\max\left\{  1,\left\vert \lambda\right\vert
\right\}
\]
which along with (\ref{xyz0}) implies%
\[
\lim_{t\rightarrow\infty}h_{k+1}\left(  t,t_{0}\right)  e_{\operatorname{Re}%
_{\mu}\gamma(t)}\left(  t,t_{0}\right)  =0
\]
as desired.
\end{proof}
\end{lemma}

\begin{theorem}
\label{thm5.5}Let $\left\{  \gamma_{i}\left(  t\right)  \right\}  _{i=1}^{n}$
be the set of conventional eigenvalues of the matrix $R(t)$ given in
(\ref{5.0.1}) and $\left\{  w_{i}\left(  t\right)  \right\}  _{i=1}^{n}$ be
the set of corresponding linearly independent dynamic eigenvectors as defined
by Lemma \ref{lem5.4}. Then, $\left\{  \gamma_{i}\left(  t\right)
,w_{i}\left(  t\right)  \right\}  _{i=1}^{n}$ is a set of dynamic eigenpairs
of $R(t)$ with the property that for each $1\leq i\leq n$ there are positive
constants $D_{i}>0$ such that%
\begin{equation}
\left\Vert w_{i}\left(  t\right)  \right\Vert \leq D_{i}%
{\displaystyle\sum\limits_{k=0}^{m_{i}-1}}
h_{k}\left(  t,t_{0}\right)  , \label{5.5}%
\end{equation}
holds where $h_{k}\left(  t,t_{0}\right)  $, $k=0,1,...,m_{i}-1$, are the
monomials defined as in (\ref{hk recursive}) and $m_{i}$ is the dimension of
the Jordan block which contains the $i^{th}$ eigenvalue, for all $1\leq i\leq
n$.
\end{theorem}

\begin{proof}
By Lemma \ref{lem5.4}, it is obvious that, $\left\{  \gamma_{i}\left(
t\right)  ,w_{i}\left(  t\right)  \right\}  _{i=1}^{n}$ is the set of
eigenpairs of $R\left(  t\right)  $. First, there exists an appropriate
$n\times n$ constant, nonsingular matrix $S$ which transforms $\Phi_{A}\left(
\delta_{+}^{T}\left(  t_{0}\right)  ,t_{0}\right)  $ to its Jordan canonical
form given by
\begin{align}
J  &  :=S^{-1}\Phi_{A}\left(  \delta_{+}^{T}\left(  t_{0}\right)
,t_{0}\right)  S\nonumber\\
&  =\left[
\begin{array}
[c]{cccc}%
J_{m_{1}}\left(  \lambda_{1}\right)  &  &  & \\
& J_{m_{2}}\left(  \lambda_{2}\right)  &  & \\
&  & \ddots & \\
&  &  & J_{m_{d}}\left(  \lambda_{d}\right)
\end{array}
\right]  _{n\times n}, \label{5.6}%
\end{align}
where $d\leq n$, $\sum_{i=1}^{d}m_{i}=n$, $\lambda_{i}$ are the eigenvalues of
$\Phi_{A}\left(  \delta_{+}^{T}\left(  t_{0}\right)  ,t_{0}\right)  $. By
utilizing above determined matrix $S$, we define the following:%
\begin{align*}
K\left(  t\right)   &  :=S^{-1}R\left(  t\right)  S\\
&  =S^{-1}\left(  \lim_{s\rightarrow t}\frac{\Phi_{A}\left(  \delta_{+}%
^{T}\left(  t_{0}\right)  ,t_{0}\right)  ^{\frac{1}{T}\Lambda(t,s)}-I}%
{\sigma\left(  t\right)  -s}\right)  S\\
&  =\lim_{s\rightarrow t}\frac{S^{-1}\Phi_{A}\left(  \delta_{+}^{T}\left(
t_{0}\right)  ,t_{0}\right)  ^{\frac{1}{T}\Lambda(t,s)}S-I}{\sigma\left(
t\right)  -s}.
\end{align*}
This along with \cite[Theorem A.6]{[3]} yields%
\[
K\left(  t\right)  =\lim_{s\rightarrow t}\frac{J^{\frac{1}{T}\Lambda(t,s)}%
-I}{\sigma\left(  t\right)  -s}.
\]

Note that, $K\left(  t\right)  $ has the block diagonal form
\[
K\left(  t\right)  =diag\left[  K_{1}\left(  t\right)  ,\ldots,K_{d}\left(
t\right)  \right]
\]
in which each $K_{i}\left(  t\right)  $ given by%
\[
K_{i}\left(  t\right)  :=\lim_{s\rightarrow t}K_{i}\left(  t\right)
:=\lim_{s\rightarrow t}\left[
\begin{array}
[c]{cccc}%
\frac{\lambda_{i}^{\frac{1}{T}\Lambda(t,s)}-1}{\sigma\left(  t\right)  -s} &
\frac{\frac{1}{T}\Lambda(t,s)\lambda_{i}^{\frac{1}{T}\Lambda(t,s)-1}}%
{(\sigma\left(  t\right)  -s)2!} & \ldots & \frac{\left(
{\displaystyle\prod\limits_{k=0}^{n-2}}
\left[  \frac{1}{T}\Lambda(t,s)-k\right]  \right)  \lambda_{i}^{\frac{1}%
{T}\Lambda(t,s)-n+1}}{\left(  n-1\right)  !\left(  \sigma\left(  t\right)
-s\right)  }\\
& \frac{\lambda_{i}^{\frac{1}{T}\Lambda(t,s)}-1}{\sigma\left(  t\right)  -s} &
\ldots & \frac{\left(
{\displaystyle\prod\limits_{k=0}^{n-3}}
\left[  \frac{1}{T}\Lambda(t,s)-k\right]  \right)  \lambda_{i}^{\frac{1}%
{T}\Lambda(t,s)-n+2}}{\left(  n-2\right)  !\left(  \sigma\left(  t\right)
-s\right)  }\\
&  & \ddots & \vdots\\
&  &  & \frac{\lambda_{i}^{\frac{1}{T}\Lambda(t,s)}-1}{\sigma\left(  t\right)
-s}%
\end{array}
\right]  _{m_{i}\times m_{i}}.
\]

It should be mentioned that, since $R\left(  t\right)  $ and $K\left(
t\right)  $ are similar, they have the same conventional eigenvalues%
\[
\gamma_{i}\left(  t\right)  =\lim_{s\rightarrow t}\left(  \frac{\lambda
_{i}^{\frac{1}{T}[\Lambda(t,s)]}-1}{\sigma\left(  t\right)  -s}\right)
\text{, }i=1,2,...,n,
\]
with corresponding multiplicities. Moreover, if we set the dynamic eigenvalues
of $K(t)$ to be same as conventional eigenvalues $\gamma_{i}\left(  t\right)
$, then the corresponding dynamic eigenvectors $\left\{  u_{i}\left(
t\right)  \right\}  _{i=1}^{n}$ of $K\left(  t\right)  $ can be given by
$u_{i}\left(  t\right)  =S^{-1}w_{i}\left(  t\right)  $.

We can prove this claim by showing that $\left\{  \gamma_{i}\left(  t\right)
,u_{i}\left(  t\right)  \right\}  _{i=1}^{n}$ is a set of dynamic eigenpairs
of $K\left(  t\right)  .$ By Definition \ref{def5.3}, we can write that%
\begin{align}
u_{i}^{\Delta}\left(  t\right)   &  =S^{-1}w_{i}^{\Delta}\left(  t\right)
\nonumber\\
&  =S^{-1}R\left(  t\right)  w_{i}\left(  t\right)  -S^{-1}\gamma_{i}\left(
t\right)  w_{i}^{\sigma}\left(  t\right) \nonumber\\
&  =K\left(  t\right)  S^{-1}w_{i}\left(  t\right)  -\gamma_{i}\left(
t\right)  S^{-1}w_{i}^{\sigma}\left(  t\right) \nonumber\\
&  =K\left(  t\right)  u_{i}\left(  t\right)  -\gamma_{i}\left(  t\right)
u_{i}^{\sigma}\left(  t\right)  , \label{5.7}%
\end{align}
for all $1\leq i\leq n$ and this proves our claim. Now, we have to show that
$u_{i}\left(  t\right)  $ satisfies (\ref{5.5}). Since $\left\{  \gamma
_{i}\left(  t\right)  ,u_{i}\left(  t\right)  \right\}  _{i=1}^{n}$ is the set
of dynamic eigenpairs of $K\left(  t\right)  ,$ it satisfies (\ref{5.7}) for
all $1\leq i\leq n.$ By choosing the $i^{th}$ block of $K\left(  t\right)  $
with dimension $m_{i}\times m_{i}$, we can construct the following linear
dynamic system:%
\begin{equation}
v^{\Delta}\left(  t\right)  =\tilde{K}_{i}\left(  t\right)  v(t)
=\lim_{s\rightarrow t}\left[
\begin{array}
[c]{ccccc}%
0 & \frac{\frac{1}{T}\Lambda(t,s)}{\left(  \sigma\left(  t\right)  -s\right)
\lambda_{i}} & \frac{\left(  \frac{1}{T}\Lambda(t,s)\right)  \left(  \frac
{1}{T}\Lambda(t,s)-1\right)  }{\left(  \sigma\left(  t\right)  -s\right)
\lambda_{i}2!} & \ldots & \frac{\left(
{\displaystyle\prod\limits_{k=0}^{n-2}}
\left[  \frac{1}{T}\Lambda(t,s)-k\right]  \right)  }{\left(  n-1\right)
!\left(  \sigma\left(  t\right)  -s\right)  \lambda_{i}^{n-1}}\\
& 0 & \frac{\frac{1}{T}\Lambda(t,s)}{\left(  \sigma\left(  t\right)
-s\right)  \lambda_{i}} &  & \frac{\left(
{\displaystyle\prod\limits_{k=0}^{n-3}}
\left[  \frac{1}{T}\Lambda(t,s)-k\right]  \right)  }{\left(  n-2\right)
!\left(  \sigma\left(  t\right)  -s\right)  \lambda_{i}^{n-2}}\\
&  & 0 & \ddots & \vdots\\
&  &  & \ddots & \frac{\frac{1}{T}\Lambda(t,s)}{\left(  \sigma\left(
t\right)  -s\right)  \lambda_{i}}\\
&  &  &  & 0
\end{array}
\right] v(t) , \label{5.8}%
\end{equation}
where $\tilde{K}_{i}\left(  t\right)  \left(  t\right)  :=K_{i}\left(
t\right)  \ominus\gamma_{i}\left(  t\right)  I.$ There are $m_{i}$ linearly
independent solutions of (\ref{5.8}). Let us denote these solutions by
$v_{i,j}\left(  t\right)  $, where $i$ corresponds to the $i^{th}$ block
matrix $K_{i}\left(  t\right)  $ and $j=1,\ldots,m_{i}$. For $1\leq i\leq d$,
we define $l_{i}=%
{\displaystyle\sum\limits_{s=0}^{i-1}}
m_{s}$, with $m_{0}=0$. Then, the form of an arbitrary $n\times1$ column
vector $u_{l_{i}+j}$ for $i\leq j\leq m$ can be given as%
\begin{equation}
u_{l_{i}+j}=[\underset{m_{1}+\ldots+m_{i-1}}{\underbrace{0,\ldots,0}%
},\underset{m_{i}}{\underbrace{v_{i,j}^{T}\left(  t\right)  }},\underset
{m_{i+1},\ldots,m_{d}}{\underbrace{0,\ldots,0}}]_{1\times n}. \label{5.9}%
\end{equation}
When we consider the all vector solutions of (\ref{5.7}), the solution of the
$n\times n$ matrix dynamic equation%
\[
U^{\Delta}\left(  t\right)  =K\left(  t\right)  U\left(  t\right)  -U^{\sigma
}\left(  t\right)  \Gamma\left(  t\right)  ,
\]
where $\Gamma\left(  t\right)  :=diag\left[  \gamma_{1}\left(  t\right)
,\ldots,\gamma_{n}\left(  t\right)  \right]  ,$ can be written as%
\begin{align*}
U\left(  t\right)   &  :=\left[  u_{1},\ldots,u_{m_{1}},\ldots,u_{\left(
\sum_{k=1}^{i-1}m_{k}\right)  },\ldots,u_{\left(  \sum_{k=1}^{i}m_{k}\right)
},\ldots,u_{\left(  \sum_{k=1}^{d}m_{k}\right)  -1},u_{n}\right] \\
&  =\left[
\begin{array}
[c]{ccc}%
\left[
\begin{array}
[c]{cccc}%
v_{1,1} & v_{1,2} & \ldots & v_{1,m_{1}}\\
& v_{1,1} & \ddots & v_{1,m_{1}-1}\\
&  & \ddots & \vdots\\
&  &  & v_{1,1}%
\end{array}
\right]  _{m_{1}\times m_{1}} &  & \\
& \ddots & \\
&  & \left[
\begin{array}
[c]{cccc}%
v_{d,1} & v_{d,2} & \ldots & v_{d,m_{d}}\\
& v_{d,1} & \ddots & v_{d,m_{d}-1}\\
&  & \ddots & \vdots\\
&  &  & v_{d,1}%
\end{array}
\right]  _{m_{d}\times m_{d}}%
\end{array}
\right]  _{n\times n}.
\end{align*}
The $m_{i}$ linearly independent solutions of (\ref{5.8}) have the form%
\begin{align*}
v_{i,1}\left(  t\right)   &  :=\left[  v_{i,m_{i}}\left(  t\right)
,0,\ldots,0\right]  _{m_{i}\times1}^{T},\\
v_{i,2}\left(  t\right)   &  :=\left[  v_{i,m_{i}-1}\left(  t\right)
,v_{i,m_{i}}\left(  t\right)  ,0,\ldots,0\right]  _{m_{i}\times1}^{T},\\
&  \vdots\\
v_{i,m_{i}}\left(  t\right)   &  :=\left[  v_{i,1}\left(  t\right)
,v_{i,2}\left(  t\right)  ,\ldots,v_{i,m_{i}-1}\left(  t\right)  ,v_{i,m_{i}%
}\left(  t\right)  \right]  _{m_{i}\times1}^{T}.
\end{align*}
Then, we have the dynamic equations%
\begin{align*}
v_{i,m_{i}}^{\Delta}\left(  t\right)   &  =0,\\
v_{i,m_{i}-1}^{\Delta}\left(  t\right)   &  =\lim_{s\rightarrow t}%
\frac{\left[  \Lambda(t,s)\right]  }{T(\sigma(t)-s)\lambda_{i}}v_{i,m_{i}%
}\left(  t\right)  ,\\
v_{i,m_{i}-2}^{\Delta}(t)  &  =\lim_{s\rightarrow t}\frac{\left(
{\displaystyle\prod\limits_{k=0}^{1}}
[\frac{1}{T}\Lambda(t,s)-k]\right)  }{2(\sigma(t)-s)\lambda_{i}^{2}}%
v_{i,m_{i}}\left(  t\right)  +\lim_{s\rightarrow t}\frac{\Lambda
(t,s)}{T(\sigma(t)-s)\lambda_{i}}v_{i,m_{i}-1}\left(  t\right)  ,\\
&  \vdots\\
v_{i,1}^{\Delta}(t)  &  =\lim_{s\rightarrow t}\frac{\left(
{\displaystyle\prod\limits_{k=0}^{m_{i}-2}}
[\frac{1}{T}\Lambda(t,s)-k]\right)  }{\left(  m_{i}-1\right)  !(\sigma
(t)-s)\lambda_{i}^{m_{i}-1}}v_{i,m_{i}}\left(  t\right) \\
&  +\lim_{s\rightarrow t}\frac{\left(
{\displaystyle\prod\limits_{k=0}^{m_{i}-3}}
[\frac{1}{T}\Lambda(t,s)-k]\right)  }{\left(  m_{i}-2\right)  !(\sigma
(t)-s)\lambda^{m_{i}-2}}v_{i,m_{i}-1}\left(  t\right)  +\\
&  \ldots+\lim_{s\rightarrow t}\frac{\left(
{\displaystyle\prod\limits_{k=0}^{1}}
[\frac{1}{T}\Lambda(t,s)-k]\right)  }{2(\sigma(t)-s)\lambda_{i}^{2}}%
v_{i,3}\left(  t\right)  +\lim_{s\rightarrow t}\frac{\Lambda(t,s)}%
{T(\sigma(t)-s)\lambda_{i}}v_{i,2}\left(  t\right)  .
\end{align*}
Moreover, we have the following solutions:%
\[
v_{i,m_{i}}\left(  t\right)  =1,\ \ v_{i,m_{i}-1}\left(  t\right)  =%
{\displaystyle\int\limits_{t_{0}}^{t}}
\lim_{s\rightarrow\tau}\frac{\Lambda(\tau,s)}{T(\sigma(\tau)-s)\lambda_{i}%
}v_{i,m_{i}}\left(  \tau\right)  \Delta\tau,
\]%
\begin{align*}
v_{i,m_{i}-2}\left(  t\right)   &  =%
{\displaystyle\int\limits_{t_{0}}^{t}}
\lim_{s\rightarrow\tau}\frac{\left(
{\displaystyle\prod\limits_{k=0}^{1}}
[\frac{1}{T}\Lambda(\tau,s)-k]\right)  }{2(\sigma(\tau)-s)\lambda_{i}^{2}%
}v_{i,m_{i}}\left(  \tau\right)  \Delta\tau+%
{\displaystyle\int\limits_{t_{0}}^{t}}
\lim_{s\rightarrow\tau}\frac{\Lambda(\tau,s)}{T(\sigma(\tau)-s)\lambda_{i}%
}v_{i,m_{i}-1}\left(  \tau\right)  \Delta\tau,\\
&  \vdots\\
v_{i,1}(t)  &  =%
{\displaystyle\int\limits_{t_{0}}^{t}}
\lim_{s\rightarrow\tau}\frac{\left(
{\displaystyle\prod\limits_{k=0}^{m_{i}-2}}
\frac{1}{T}\Lambda(\tau,s)-k]\right)  }{\left(  m_{i}-1\right)  !(\sigma
(\tau)-s)\lambda_{i}^{m_{i}-1}}v_{i,m_{i}}\left(  \tau\right)  \Delta\tau\\
&  +%
{\displaystyle\int\limits_{t_{0}}^{t}}
\lim_{s\rightarrow\tau}\frac{\left(
{\displaystyle\prod\limits_{k=0}^{m_{i}-3}}
\frac{1}{T}\Lambda(\tau,s)-k]\right)  }{\left(  m_{i}-2\right)  !(\sigma
(\tau)-s)\lambda_{i}^{m_{i}-2}}v_{i,m_{i}-1}(\tau)\Delta\tau+\\
&  \ldots+%
{\displaystyle\int\limits_{t_{0}}^{t}}
\lim_{s\rightarrow\tau}\frac{\Lambda(\tau,s)}{T(\sigma(\tau)-s)\lambda_{i}%
}v_{i,2}\left(  \tau\right)  \Delta\tau.
\end{align*}
Then we can show that each $v_{i,j}$ is bounded. There exist constants
$B_{i,j},$ $i=1,\ldots,d$ and $j=1,\ldots,m_{i},$ such that%
\begin{align*}
\left\vert v_{i,m_{i}}\left(  t\right)  \right\vert  &  =1\leq B_{i,m_{i}%
}h_{0}\left(  t,t_{0}\right)  =B_{i,m_{i}},\\
\left\vert v_{i,m_{i}-1}\left(  t\right)  \right\vert  &  \leq%
{\displaystyle\int\limits_{t_{0}}^{t}}
\lim_{s\rightarrow\tau}\left(  \frac{\Lambda(\tau,s)}{T(\sigma(\tau
)-s)\lambda_{i}}\right)  v_{i,m_{i}}\left(  \tau\right)  \Delta\tau\leq
\frac{1}{T\lambda_{i}}%
{\displaystyle\int\limits_{t_{0}}^{t}}
\lim_{s\rightarrow\tau}\left(  \frac{\Lambda(\tau,s)}{\sigma(\tau)-s}\right)
h_{0}\left(  \tau,t_{0}\right)  \Delta\tau\\
&  \leq\frac{h_{1}\left(  t,t_{0}\right)  }{T\lambda_{i}}\leq B_{i,m_{i}%
-1}h_{1}\left(  t,t_{0}\right)  ,\\
\left\vert v_{i,m_{i}-2}\left(  t\right)  \right\vert  &  \leq%
{\displaystyle\int\limits_{t_{0}}^{t}}
\left\vert \lim_{s\rightarrow\tau}\frac{\left(
{\displaystyle\prod\limits_{k=0}^{1}}
[\frac{1}{T}\Lambda(\tau,s)-k]\right)  }{2(\sigma(\tau)-s)\lambda_{i}^{2}%
}\right\vert v_{i,m_{i}}\left(  \tau\right)  \Delta\tau\\
&  +%
{\displaystyle\int\limits_{t_{0}}^{t}}
\lim_{s\rightarrow\tau}\left(  \frac{\Lambda(\tau,s)}{T(\sigma(\tau
)-s)\lambda_{i}}\right)  v_{i,m_{i}-1}\left(  \tau\right)  \Delta\tau.
\end{align*}
Since%
\[
0\leq\Theta\left(  \sigma\left(  \tau\right)  \right)  -\Theta\left(
s\right)  \leq T\text{ as }s\rightarrow\tau,
\]
we get%
\[
\left\vert \frac{1}{T}\Lambda(\tau,s)-k\right\vert \leq k\text{ as
}s\rightarrow\tau\text{ for }k=1,2,...\text{.}%
\]
Then%
\begin{align*}
\left\vert v_{m_{i}-2}\left(  t\right)  \right\vert  &  \leq\frac{1}%
{2T\lambda_{i}^{2}}%
{\displaystyle\int\limits_{t_{0}}^{t}}
\lim_{s\rightarrow\tau}\left(  \frac{\Lambda(\tau,s)}{\sigma(\tau)-s}\right)
h_{0}\left(  \tau,t_{0}\right)  \Delta\tau\\
&  +\frac{1}{T^{2}\lambda_{i}^{2}}%
{\displaystyle\int\limits_{t_{0}}^{t}}
\lim_{s\rightarrow\tau}\left(  \frac{\Lambda(\tau,s)}{\sigma(\tau)-s}\right)
h_{1}\left(  \tau,t_{0}\right)  \Delta\tau\\
&  =\frac{h_{1}\left(  t,t_{0}\right)  }{2T\lambda_{i}^{2}}+\frac{h_{2}\left(
t,t_{0}\right)  }{T^{2}\lambda_{i}^{2}}\\
&  \leq B_{i,m_{i}-2}%
{\displaystyle\sum\limits_{j=1}^{2}}
h_{j}\left(  t,t_{0}\right) \\
&  \vdots\\
\left\vert v_{1}\right\vert  &  \leq B_{i,1}%
{\displaystyle\sum\limits_{j=1}^{m_{i}-1}}
h_{j}\left(  t,t_{0}\right)  .
\end{align*}
If we set $\beta_{i}:=\max_{j=1,\ldots,m_{i}}\left\{  B_{i,j}\right\}  $ for
each $1\leq i\leq d$, we obtain%
\[
\left\Vert u_{l_{i}+j}\left(  t\right)  \right\Vert \leq\beta_{i}%
{\displaystyle\sum\limits_{k=0}^{m_{i}-1}}
h_{k}\left(  t,t_{0}\right)
\]
for $1\leq i\leq d$ and $j=1,2,...,m_{i}$. Since $w_{i}=Su_{i}$ we have
\begin{align*}
\left\Vert w_{i}\left(  t\right)  \right\Vert  &  =\left\Vert Su_{i}\left(
t\right)  \right\Vert \leq\left\Vert S\right\Vert \beta_{i}%
{\displaystyle\sum\limits_{k=0}^{m_{i}-1}}
h_{k}\left(  t,t_{0}\right) \\
&  =D_{i}%
{\displaystyle\sum\limits_{k=0}^{m_{i}-1}}
h_{k}\left(  t,t_{0}\right)  ,
\end{align*}
where $D_{i}:=\left\Vert S\right\Vert \beta_{i}$, for all $1\leq i\leq n$. The
proof is complete.
\end{proof}

\begin{definition}
\label{def5.8}\cite[Definition 7.8]{[3]} Let $\mathbb{C}_{\mu}:=\left\{
z\in\mathbb{C}:z\neq-\frac{1}{\mu(t)}\right\}  .$ Given an element
$t\in\mathbb{T}^{k}$ with $\mu\left(  t\right)  >0,$ the Hilger circle is
defined by%
\[
\mathcal{H}_{t}:=\left\{  z\in\mathbb{C}_{\mu}:\operatorname{Re}_{\mu
}(z)<0\right\}  .
\]
If $\mu\left(  t\right)  =0,$ Hilger circle becomes
\[
\mathcal{H}_{t}:=\left\{  z\in\mathbb{C}:\operatorname{Re}(z)<0\right\}  .
\]

\end{definition}

Now, we can state the main stability theorem. This theorem shows strong
relationship between the stability results of the $T$-periodic time varying
linear dynamic system (\ref{5.0}) and the eigenvalues of the corresponding
time varying linear dynamic system (\ref{5.1}).

\begin{theorem}
[Floquet stability theorem]\label{thm5.9}Let $\mathbb{T}$ be a periodic time
scale in shifts that is unbounded above$.$ We get the following stability
results of the solutions of the system (\ref{5.0}) based on the eigenvalues
$\{\gamma_{i}(t)\}_{i=1}^{n}$ of system (\ref{5.1}):

\begin{enumerate}
\item \label{condition 1}If there is a positive constant $H$ such that $\ $%
\begin{equation}
\inf_{t\in\lbrack H,\infty)_{\mathbb{T}}}\left[  -\left(  \lim_{s\rightarrow
t}\left(  \frac{\Lambda(t,s)}{\sigma(t)-s}\right)  \right)  ^{-1}%
\operatorname{Re}_{\mu}\gamma_{i}\left(  t\right)  \right]  >0
\label{inf cond}%
\end{equation}
for all $i=1,\ldots,n$, then the system (\ref{5.0}) is asymptotically stable.
Moreover, if there are positive constants $H$ and $\varepsilon$ such that
(\ref{inf cond}) and%
\begin{equation}
-\operatorname{Re}_{\mu}\gamma_{i}\left(  t\right)  \geq\varepsilon
\label{cond eps}%
\end{equation}
for all $t\in\lbrack H,\infty)_{\mathbb{T}}$ and all $i=1,\ldots,n$, then the
system (\ref{5.0}) is exponentially stable.

\item If there is a positive constant $H$ such that $\ $%
\begin{equation}
\inf_{t\in\lbrack H,\infty)_{\mathbb{T}}}\left[  -\left(  \lim_{s\rightarrow
t}\left(  \frac{\Lambda(t,s)}{\sigma(t)-s}\right)  \right)  ^{-1}%
\operatorname{Re}_{\mu}\gamma_{i}\left(  t\right)  \right]  \geq0
\end{equation}
for all $i=1,\ldots,n$, and if, for each characteristic exponent with%
\[
\operatorname{Re}_{\mu}\left(  \gamma_{i}\left(  t\right)  \right)  =0\text{
for all }t\in\lbrack H,\infty)_{\mathbb{T}},
\]
the algebraic multiplicity equals the geometric multiplicity, then the system
(\ref{5.0}) is stable; otherwise the system (\ref{5.0}) is unstable.

\item If there exists a number $H\in\mathbb{R}$ such that%
\[
\operatorname{Re}_{\mu}\left(  \gamma_{i}\left(  t\right)  \right)  >0
\]
for all $t\in\lbrack H,\infty)_{\mathbb{T}}$ and some $i=1,\ldots,n$, then the
system (\ref{5.0}) is unstable.
\end{enumerate}
\end{theorem}

\begin{proof}
Let $e_{R}\left(  t,t_{0}\right)  $ be the transition matrix of the system
(\ref{5.1}) and $R\left(  t\right)  $ be defined as in (\ref{5.0.1}). Given
the conventional eigenvalues $\left\{  \gamma_{i}\left(  t\right)  \right\}
_{i=1}^{n}$ of $R\left(  t\right)  $, we can define the set of dynamic
eigenpairs $\left\{  \gamma_{i}\left(  t\right)  ,w_{i}\left(  t\right)
\right\}  _{i=1}^{n}$ and from Theorem \ref{thm5.5}, the dynamic eigenvector
$w_{i}\left(  t\right)  $ satisfies (\ref{5.5}). Moreover, let us define
$W(t)$ as the following:%
\begin{equation}
W(t)=e_{R}\left(  t,\tau\right)  e_{\ominus\Xi}\left(  t,\tau\right)
\label{5.9.1}%
\end{equation}
and we have%
\begin{equation}
e_{R}\left(  t,\tau\right)  =W(t)e_{\Xi}\left(  t,\tau\right)  , \label{5.9.2}%
\end{equation}
where $\tau\in\mathbb{T}$ and $\Xi\left(  t\right)  $ is given as in Lemma
\ref{lem5.4}. Employing (\ref{5.9.2}), we can write that%
\begin{equation}
e_{R}\left(  \tau,t_{0}\right)  =e_{\Xi}\left(  \tau,t_{0}\right)
W^{-1}(t_{0}). \label{5.9.3}%
\end{equation}
By combining (\ref{5.9.2}) and (\ref{5.9.3}), the transition matrix of the
system (\ref{5.1}) can be represented by%
\begin{equation}
e_{R}\left(  t,t_{0}\right)  =W\left(  t\right)  e_{\Xi}\left(  t,t_{0}%
\right)  W^{-1}\left(  t_{0}\right)  , \label{5.10}%
\end{equation}
where $W\left(  t\right)  :=[w_{1}\left(  t\right)  ,w_{2}\left(  t\right)
,\ldots,w_{n}\left(  t\right)  ]$. Furthermore, we can denote the matrix
$W^{-1}\left(  t_{0}\right)  $ as follows:%
\[
W^{-1}\left(  t_{0}\right)  =\left[
\begin{array}
[c]{c}%
v_{1}^{T}\left(  t_{0}\right) \\
v_{2}^{T}\left(  t_{0}\right) \\
\vdots\\
v_{n}^{T}\left(  t_{0}\right)
\end{array}
\right]  .
\]

Since $\Xi\left(  t\right)  $ is a diagonal matrix, we can write (\ref{5.10})
as
\begin{equation}
e_{R}\left(  t,t_{0}\right)  =%
{\displaystyle\sum\limits_{i=1}^{n}}
e_{\gamma_{i}}\left(  t,t_{0}\right)  W\left(  t\right)  F_{i}W^{-1}\left(
t_{0}\right)  , \label{5.11}%
\end{equation}
where $F_{i}:=\delta_{i,j}$ is $n\times n$ matrix. Using $v_{i}^{T}\left(
t\right)  w_{j}\left(  t\right)  =\delta_{i,j}$ for all $t\in\mathbb{T}$, we
rewrite $F_{i}\ $\ as follows:%
\begin{equation}
F_{i}=W^{-1}\left(  t\right)  \left[  0,\ldots,0,w_{i}\left(  t\right)
,0,\ldots,0\right]  . \label{5.12}%
\end{equation}
By means of (\ref{5.11}) and (\ref{5.12}) we have%
\[
e_{R}\left(  t,t_{0}\right)  =%
{\displaystyle\sum\limits_{i=1}^{n}}
e_{\gamma_{i}}\left(  t,t_{0}\right)  w_{i}\left(  t\right)  v_{i}^{T}\left(
t_{0}\right)  =%
{\displaystyle\sum\limits_{i=1}^{n}}
\chi_{i}\left(  t\right)  v_{i}^{T}\left(  t_{0}\right)  ,
\]
where $\chi_{i}\left(  t\right)  $ is mode vector of system (\ref{5.1}).

\begin{case}
\label{case1}By (\ref{5.3}), for each $1\leq i\leq n,$ we can write that%
\begin{align*}
\left\Vert \chi_{i}\left(  t\right)  \right\Vert  &  \leq D_{i}%
{\displaystyle\sum\limits_{k=0}^{d_{i}-1}}
h_{k}\left(  t,t_{0}\right)  \left\vert e_{\gamma_{i}}\left(  t,t_{0}\right)
\right\vert \\
&  \leq D_{i}%
{\displaystyle\sum\limits_{k=0}^{d_{i}-1}}
h_{k}\left(  t,t_{0}\right)  e_{\operatorname{Re}_{\mu}(_{\gamma_{i}})}\left(
t,t_{0}\right)
\end{align*}
where $D_{i}$ is as in Theorem \ref{thm5.5}, $d_{i}$ represents the dimension
of the Jordan block which contains $i^{th}$ eigenvalue of $R\left(  t\right)
$.Using Lemma \ref{lem5.6} we get%
\[
\lim_{t\rightarrow\infty}h_{k}\left(  t,t_{0}\right)  e_{\operatorname{Re}%
_{\mu}(_{\gamma_{i}})}\left(  t,t_{0}\right)  =0
\]
for each $1\leq i\leq n$ and all $k=1,2,...,d_{i}-1$. This along with Theorem
\ref{thm5.4.1} implies that (\ref{5.1}) is asymptotically stable. By Theorem
\ref{thm2}, since the solutions of (\ref{5.0}) and (\ref{5.1}) are related by
Lyapunov transformation, we can state that solution of (\ref{5.0}) is
asymptotically stable. For the second part, we first write%
\begin{align}
\left\Vert \chi_{i}\left(  t\right)  \right\Vert  &  \leq D_{i}%
{\displaystyle\sum\limits_{k=0}^{d_{i}-1}}
h_{k}\left(  t,t_{0}\right)  \left\vert e_{\gamma_{i}}\left(  t,t_{0}\right)
\right\vert \nonumber\\
&  \leq D_{i}%
{\displaystyle\sum\limits_{k=0}^{d_{i}-1}}
h_{k}\left(  t,t_{0}\right)  e_{\operatorname{Re}_{\mu}(_{\gamma_{i}}%
)\oplus\varepsilon}\left(  t,t_{0}\right)  e_{\ominus\varepsilon}\left(
t,t_{0}\right)  . \label{5.13}%
\end{align}
If (\ref{cond eps}) holds, then $\operatorname{Re}_{\mu}\left(  \gamma
_{i}\oplus\varepsilon\right)  $ satisfies (\ref{infimum condtion}). Hence, by
Lemma \ref{lem5.6} the term $h_{k}\left(  t,t_{0}\right)  e_{\operatorname{Re}%
_{\mu}(_{\gamma_{i}})\oplus\varepsilon}\left(  t,t_{0}\right)  $ converges to
zero as $t\rightarrow\infty$. That is, there is an upper bound $C_{\varepsilon
}$ for the sum $%
{\textstyle\sum\limits_{k=0}^{d_{i}-1}}
h_{k}\left(  t,t_{0}\right)  e_{\operatorname{Re}_{\mu}(_{\gamma_{i}}%
)\oplus\varepsilon}\left(  t,t_{0}\right)  $. This along with (\ref{5.13})
yields
\[
\left\Vert \chi_{i}\left(  t\right)  \right\Vert \leq D_{i}C_{\varepsilon
}e_{\ominus\varepsilon}\left(  t,t_{0}\right)  .
\]
Thus,Theorem \ref{thm5.4.1} implies that (\ref{5.1}) is exponentially stable.
Using the above given argument (\ref{5.0}) is exponentially stable.
\end{case}

\begin{case}
\label{case2}Assume that $\operatorname{Re}_{\mu}\left[  \gamma_{k}\left(
t\right)  \right]  =0$ for some $1\leq k\leq n$ with equal algebraic and
geometric multiplicities corresponding to $\gamma_{k}(t).$ Then the Jordan
block of $\gamma_{k}(t)$ is $1\times1$ and this implies
\[
\chi_{k}\left(  t\right)  =\beta_{k}e_{\gamma_{k}}\left(  t,t_{0}\right)  .
\]
Thus,%
\begin{align*}
\lim_{t\rightarrow\infty}\left\Vert \chi_{k}\left(  t\right)  \right\Vert  &
\leq\lim_{t\rightarrow\infty}\beta_{k}\left\vert e_{\gamma_{k}}\left(
t,t_{0}\right)  \right\vert \\
&  \leq\lim_{t\rightarrow\infty}\beta_{k}e_{\operatorname{Re}_{\mu}%
(_{\gamma_{k}})}\left(  t,t_{0}\right) \\
&  =0.
\end{align*}
By Theorem \ref{thm5.4.1}, the system (\ref{5.1}) is stable. By Theorem
\ref{thm2}, the solutions of (\ref{5.0}) and (\ref{5.1}) are related by
Lyapunov transformation. This implies that the system (\ref{5.0}) is stable.
\end{case}

\begin{case}
\label{case3}Suppose that $\operatorname{Re}_{\mu}(\gamma_{i}\left(  t\right)
)>0$ for some $i=1,\ldots,n.$ Then, we have%
\[
\lim_{t\rightarrow\infty}\left\Vert e_{R}\left(  t,t_{0}\right)  \right\Vert
=\infty,
\]
and by the relationship between solutions of (\ref{5.0}) and (\ref{5.1}), we
can write that%
\[
\lim_{t\rightarrow\infty}\left\Vert \Phi_{A}\left(  t,t_{0}\right)
\right\Vert =\infty.
\]
Therefore, (\ref{5.0}) is unstable.
\end{case}
\end{proof}

\begin{remark}
In the case when the time scale is additive periodic, Theorem \ref{thm5.9}
gives its additive counterpart \cite[Theorem 7.9]{[3]}. For an additive time
scale the graininess function $\mu(t)$ is bounded above by the period of the
time scale. However, this is not true in general for the times scales that are
periodic in shifts. The highlight of Theorem \ref{thm5.9} is to rule out
strong restriction that obliges the time scale to be additive periodic. Hence,
unlike \cite[Theorem 7.9]{[3]} our stability theorem (i.e. Theorem
\ref{thm5.9} ) is valid for $q$-difference systems.
\end{remark}

We can state the following corollary as a consequence of Theorem \ref{thm5.9}.

\begin{corollary}
\label{cor5.10}Consider the $T$-periodic linear dynamic system (\ref{4.1});

\begin{enumerate}
\item If all Floquet multipliers have modulus less than $1$, then the system
(\ref{4.1}) is exponentially stable;

\item If all Floquet multipliers have modulus less than $1$ or equal to $1,$
and if, for each Floquet multiplier with modulus less than $1$, the algebraic
multiplicity equals to geometric multiplicity, then the system (\ref{4.1}) is
stable, otherwise the system (\ref{4.1}) is unstable, growing at rates of
generalized polynomials of $t;$

\item If at least one Floquet multiplier has modulus greater than $1$, then the
system (\ref{4.1}) is unstable.
\end{enumerate}
\end{corollary}

Now, we can revisit our examples to make stability analysis:

\begin{example}
Let $\mathbb{T}=\overline{q^{\mathbb{Z}}},$ $q>1$ and consider the following
system%
\begin{align}
x^{\Delta}\left(  t\right)   &  =A(t)x\left(  t\right) \nonumber\\
&  =\left[
\begin{array}
[c]{cc}%
\frac{1}{t} & 0\\
0 & \frac{1}{t}%
\end{array}
\right]  x\left(  t\right)  . \label{last}%
\end{align}
As we did in Example \ref{examp1} we obtain $R\left(  t\right)  $ as follows:%
\[
R\left(  t\right)  =%
\begin{bmatrix}
\frac{1}{t} & 0\\
0 & \frac{1}{t}%
\end{bmatrix}
.
\]
Then $R\left(  t\right)  $ has eigenvalues $\gamma_{1,2}(t)=1/t$ and
\begin{align*}
\operatorname{Re}_{\mu}(\gamma_{1,2}\left(  t\right)  )  &  =\frac{\left\vert
\mu(t)\gamma_{1,2}\left(  t\right)  +1\right\vert -1}{\mu(t)}\\
&  =\frac{\left\vert (qt-t)\frac{1}{t}+1\right\vert -1}{qt-t}\\
&  =\frac{q-1}{qt-t}\\
&  =\frac{1}{t}>0.
\end{align*}
Thus, we can conclude by the preceding theorem that the system (\ref{last}) is unstable.
\end{example}

\end{document}